\documentclass[preprint,10pt]{elsarticle}
\usepackage{amsmath}
\usepackage{amssymb,amsfonts}
\usepackage{xypic}
\xyoption{all}

\newtheorem{theorem}{Theorem}[section]
\newtheorem{definition}[theorem]{Definition}

\newtheorem{corollary}[theorem]{Corollary}
\newtheorem{lemma}[theorem]{Lemma}
\newtheorem{remark}[theorem]{Remark}
\newtheorem{example}[theorem]{Example}
\newtheorem{fact}[theorem]{Fact}
\newproof{proof}{Proof}

\newcommand{\lv}{\ensuremath{\mathbf{v}}}
\newcommand{\y}{\ensuremath{\mathbf{y}}}
\newcommand{\x}{\ensuremath{\mathbf{x}}}
\newcommand{\yp}{\ensuremath{\mathbf{y'}}}
\newcommand{\z}{\ensuremath{\mathbf{z}}}

\newcommand{\set}{\ensuremath{\mathbf{Set}}}
\newcommand{\sset}{\ensuremath{\mathbf{Simpl.Set}}}
\newcommand{\bsset}{\ensuremath{\mathbf{Bisimpl.Set}}}
\newcommand{\cat}{\ensuremath{\mathbf{Cat}}}
\newcommand{\ner}{\ensuremath{\mathrm{N}}}
\newcommand{\gner}{\ensuremath{\Delta}}
\newcommand{\class}{\ensuremath{\mathrm{B}}}
\newcommand{\diag}{\ensuremath{\mathrm{diag}}}
\newcommand{\m}{\ensuremath{\mathcal{M}}}
\newcommand{\f}{\ensuremath{\mathcal{F}}}
\newcommand{\dcat}{\bf{2}\ensuremath{\mathbf{Cat}}}
\newcommand{\Lfunc}{\ensuremath{\mathrm{laxFunc}}}
\newcommand{\A}{\ensuremath{\mathcal{A}}}
\newcommand{\C}{\ensuremath{\mathcal{C}}}
\newcommand{\B}{\ensuremath{\mathcal{B}}}
\def\comma#1#2{#1/\hspace{-3pt}/ #2}
 \journal{Journal of Pure and Applied Algebra}

\begin{document}

\begin{frontmatter}

\title{Homotopy fibre sequences induced by $2$-functors}
%%\tnotetext[t]{}

\author{A. M. Cegarra\corref{cor}}
\ead{acegarra@ugr.es}
\address{Departamento de \'{A}lgebra,  Universidad de
Granada, 18071 Granada, Spain}

\begin{abstract}
This paper contains some contributions to the study of the relationship between 2-categories and the homotopy types of their classifying spaces. Mainly, generalizations are given of both  Quillen's Theorem B and Thomason's Homotopy Colimit Theorem to 2-functors.
\end{abstract}

\end{frontmatter}

\section{Introduction and summary}

The construction of nerves and classifying spaces of categorical structures has become an essential part of the machinery in algebraic topology and algebraic K-theory. This paper provides a contribution to the study of classifying spaces of $2$-categories and, in particular, of monoidal categories.

To help motivate the reader, and to establish the setting for our discussions,  let us briefly recall that it was Grothendieck \cite{groth, grothendieck} who first associated a simplicial set $\ner\C$ to a small category $\C$, calling it its \emph{nerve}. The set of $n$-simplices
$$
\ner_n\C = \bigsqcup_{(x_0,\ldots,x_n)\in \mbox{\scriptsize Ob}\C^{n+1}}\hspace{-0.3cm}
\C(x_1,x_0)\times\C(x_2,x_1)\times\cdots\times\C(x_n,x_{n-1})
$$ consists of length $n$ sequences of composable morphisms in $\C$.   Milnor's realization \cite{milnor} of its nerve is the {\em classifying space} of the category, $\class\C=|\ner\C|$. We can stress the historical relevance of this construction by noting  that, in Quillen's development of higher algebraic K-theory \cite{quillen}, K-groups are defined by taking homotopy groups of classifying spaces of certain categories. When a monoid (or group) ${\mathcal M}$ is regarded as a category with only
one object, then $\class{\mathcal M}$ is its classifying space in the traditional sense. Therefore, many weak homotopy types
thus occur, since every path-connected space has the weak homotopy type  of the classifying space of a monoid, as it was proved by McDuff \cite[Theorem 1]{macd} (see also \cite[Theorem 3.5]{fiedorowicz}). Moreover, any CW-complex is homotopy equivalent to the classifying space of a small category,
as Illusie showed in \cite[Theorem 3.3]{illusie}: The category of simplices $\int_\Delta {\hspace{-3pt}}S$, of a simplicial set $S$, has as objects pairs
$(p,x)$ where $p\geq 0$ and $x$ is a $p$-simplex of $S$; and arrow $\xi\!:(p,x)\rightarrow (q,y)$ is an arrow
$\xi\!:[p]\to [q]$ in $\Delta$ with the property $x=\xi^*y$. Then there exists a homotopy equivalence $|S|\simeq
\class\!\int_\Delta {\hspace{-2pt}}S$ between the geometric realization of $S$ and the classifying space of $\int_\Delta{\hspace{-3pt}}S$ (this result is,
in fact, a very particular case of the homotopy colimit theorem of Thomason \cite{thomason}). Then, by \cite[Theorem 4]{milnor}, if $X$ is any CW-complex and we
take $S=SX$, the total singular complex of $X$,  it follows that  $X\simeq |SX|\simeq \class\!\int_\Delta{\hspace{-3pt}} SX$.

  In \cite{segal68}, Segal extended Milnor's realization process to simplicial
topological spaces. He observed that, if \C\ is a topological category, then \ner\C\ is, in a
natural way, a simplicial space and he defines the \emph{classifying space} \class\C\ of a
topological category \C\ as the realization of the simplicial space $\ner\C$. This general
construction given by Segal provides, for instance, the definition of classifying spaces of
2-categories. A 2-category $\C$ is a category endowed with categorical hom-sets $\C(x,y)$, for any
pair of objects, in such a way that composition is a functor $\C(y,z)\times \C(x,y)\rightarrow \C(x,z)$ satisfying the usual identities.
 By replacing the hom-categories $\C(x,y)$ by their classifying spaces $\class\C(x,y)$, one
obtains a topological category with a discrete space of objects, of which Segal's realization is the
classifying space \class\C\ of the 2-category.  Thus, $$\class \C:=|\class\ner\C|\cong |\diag\ner\ner\C|,$$where $\diag\ner\ner\C$ is the diagonal of the bisimplicial set obtained from the simplicial category $\ner \C$ by replacing each category $\C(x,y)$ by its nerve (i.e., the {\em double nerve} of the 2-category).  For instance, the classifying space of a (strict) monoidal category  $(\m,\otimes)$ is  the classifying space of the one object $2$-category, with category of endomorphisms $\m$, that it defines.

The category $\dcat$ of small 2-categories and 2-functors has a Quillen model structure \cite{tonks},  such that the functor $\C\mapsto\class\C$ induces an equivalence between the corresponding homotopy category of 2-categories and the ordinary homotopy category of CW-complexes. By this correspondence, 2-groupoids correspond to spaces whose homotopy groups $\pi_n$ are trivial for $n > 2$ \cite{moerdijk-svensson}. From this point of view the use of 2-categories and their classifying spaces in homotopy theory goes back to Whitehead \cite{whi} and Mac Lane-Whitehead \cite{mac} since 2-groupoids with only one object (= strict 2-groups, in the
terminology of Baez \cite{baez}) are the same as crossed modules (this observation is attributed to Verdier  in \cite{b-s}). But, beyond homotopy theory,  the use of classifying spaces of $2$-categories has also shown its relevance in several other mathematical contexts such as in algebraic K-theory \cite{hinich}, conformal field theory \cite{til2}, or in the study of geometric structures on low-dimensional manifolds \cite{til}.

This paper contributes to clarifying several relationships between 2-categories and the homotopy types of their classifying spaces. Here, we deal with questions such as, when do 2-functors induce homotopy equivalences or homotopy cartesian squares of classifying spaces? In fact, research work aiming to answer that question was started by the author with Bullejos in \cite{b-c}, a
paper to which this one is a sequel,  and where  a generalization to 2-functors of Quillen's Theorem A \cite{quillen}  was shown as a primary outcome. That paper relied heavily on the fact that realizations of Duskin-Street's geometric nerves \cite{duskin,street} yield classifying spaces for 2-categories \cite[Theorem 1] {b-c}. In this paper we mainly state and prove an extension of the relevant and well-known Quillen's Theorem B \cite{quillen} to 2-categories by showing, under reasonable necessary conditions, a category-theoretical interpretation of the homotopy fibres of the realization map $\class F:\class \B\to\class \C$, of a 2-functor $F:\B\to\C$.
More precisely, in Theorem \ref{B} we replace the concept of homotopy fibre category of a functor by Gray's notion \cite[\S 3.1]{gray} of homotopy fibre 2-category $\comma{z}F$ of a 2-functor $F:\B\to\C$ at an object $z$ of $\C$ (the double bar notation $\comma{}$ avoids confusion with the homotopy fibre category of the underlying functor, see \S 3 for details). Then,   we prove the existence of induced homotopy fibre sequences $$\class (\comma{z}F)\to \class \B\to\class\C,$$  whenever the induced maps $\class(\comma{z_0}F)\to \class(\comma{z_1}F)$, for the different morphisms (1-cells) $z_1\to z_0$  of $\C$, are homotopy equivalences. This says  that the name ``homotopy fibre 2-category" was well-chosen, since the classifying spaces of these homotopy fibres are the homotopy fibres of the map of classifying spaces. Clearly if the homotopy fibre 2-categories are contractible then $\class F$ is a homotopy equivalence, and Theorem A for 2-functors in \cite{b-c} becomes an immediate corollary.

When both $\B$ and $\C$ are small categories, regarded as discrete 2-categories, one obtains the ordinary Theorem B, and the methods used in the proof of Theorem \ref{B} we give follow along similar lines to those used by Goerss and Jardine in \cite[\S IV, 5.2]{g-j} for proving Quillen's theorem, though the generalization to 2-categories is highly nontrivial.

Our result is applied to the homotopy theory of lax functors ${\!\f:\C^{o}\rightsquigarrow \dcat}$, where $\C$ is any 2-category, and hence to acting monoidal categories. The application is carried out through an enriched Grothendieck construction $\int_\C\!\f$, which is actually a special case of the one considered by Bakovi\'{c} in \cite[\S 4]{bak} and, for the case where $\C$ is a category, it is a special case of the  ones given by Tamaki in \cite[\S 3]{tama} and by Carrasco, Cegarra and Garz\'on in \cite[\S 3]{c-c-g} (see also the construction $\int_\Gamma$ in \cite[Theorem 3.3]{c-g-o}, where $\Gamma$ is a group). For any  lax functor ${\!\f:\C^{o}\rightsquigarrow \dcat}$, $\int_\C\!\f$ is a 2-category that assembles all 2-categories $\!\f_z$, $z\in \text{Ob}\C$, and, in Theorem \ref{hc2},  when every map $\class \f_{z_0}\to\class\f_{z_1}$ induced by a morphism $z_1\to z_0$ of $\C$ is a homotopy equivalence, we prove the existence of homotopy fibre sequences $$\class\f_z\to\xymatrix{\class\!\int_\C\!\f}\to \class\C,$$  for the different objects $z$ of $\C$.

The usual  Grothendieck construction \cite{grothendieck} on a lax functor ${\!\f:\C^{o}\rightsquigarrow \cat}$, for $\C$ a category, underlies
our 2-categorical construction $\int_\C\!\f$. Recall also  that the well-known Homotopy Colimit Theorem by Thomason \cite{thomason} establishes that  the Grothendieck construction on a diagram of categories is actually a categorical model for the homotopy type of the homotopy colimit of the diagram
of categories. The notion of homotopy colimit has been well generalized in the literature to 2-functors $\f\!:\C^{o}\to
\cat$ (see \cite[2.2]{hinich}, for example), where $\C$ is any 2-category.  Our Theorem \ref{hct} states two generalizations of Thomason's theorem both to
2-diagrams of categories, that is to 2-functors  $\!\f:\C^{o}\to \cat$, with $\C$ a 2-category,  and  to diagrams of 2-categories, that is to functors  $\!\f:\C^{o}\to
\dcat$, with $\C$ a category, by showing the existence of respective homotopy equivalences
$$\xymatrix{\class \,\mbox{hocolim}_\C\f\simeq
\class\!\int_\C\!\f.}$$

The plan of this paper is, briefly, as follows. After this historical and motivating introductory Section 1, the paper is organized in three sections. Section 2 contains a minimal amount of notation as well as various auxiliary statements on classifying spaces of 2-categories. In Section 3 we state and prove the main result of the paper, Theorem \ref{B}, with the generalization to 2-functors of Quillen's Theorem B which, in Section 4, is applied to the study of 2-diagrams of 2-categories by means of the higher Grothendieck construction.

\section{Preliminaries and notations}\label{SP}

For the general background on 2-categories used in this paper, we refer to \cite{borceux}, \cite{maclane} and \cite{street2}, and on
simplicial sets to \cite{may67}, \cite{quillen} and, mainly, to \cite{g-j}.

The {\em simplicial category} is denoted by $\Delta$. It has as objects the ordered sets $[n]=\{0,\ldots,n\}$, $n\geq
0$, and as arrows the (weakly) monotone maps $\alpha:[n]\rightarrow[m]$. This category is generated by the directed
graph with edges all maps
$$
\xymatrix@C=16pt{ [n+1]\ar[r]^-{\textstyle \sigma_i} & [n] & [n-1] \ar[l]_{\textstyle \delta_i}},\qquad 0\leq i\leq n,
$$
where
$$
\sigma_i(j)=\left\{
\begin{array}{lll}
  j & \text{ if } &j\leq i \\
  j-1 &\text{ if } &j> i
\end{array}
\right. \qquad\mbox{and}\qquad \delta_i(j)=\left\{
\begin{array}{lll}
  j &\text{ if }& j < i \\
  j+1 &\text{ if }& j\geq i \; .
\end{array}
\right.
$$

Segal's {\em geometric realization} functor \cite{segal68},  for
simplicial spaces $S\!:\Delta^{\!o}\to  \mathbf{Top}$,  is denoted
by $S\mapsto |S|$. For instance, by regarding a set as a discrete
space, the (Milnor's, \cite{milnor}) geometric realization of a
simplicial set $S:\Delta^{\!o}\to \set$ is $|S|$.

The category of small categories is $\cat$, and $\ner:\cat \to \sset$ denotes the {\em nerve} functor to the category
of simplicial sets.
  Any ordered set $[n]$ is
considered as a category with only one morphism  $j\rightarrow i$ when $0\leq i\leq j\leq n$, and the nerve $$\ner \C:
\Delta^{\!o}\to\set,$$ of a category \C, is the simplicial set with $\ner_n\C=\mbox{Func}([n],\C)$, the set of all functors
${\x:[n]\to \C}$, which usually we write by $$\xymatrix{\x= \big(x_i&\ar[l]_(0.6){\textstyle x_{i,j}}x_j\big)_{\hspace{-5pt}\scriptsize
\begin{array}{l}0\leq i\leq j\leq n.\end{array}}}$$ In other words, for $n>0$, an $n$-simplex is
  a string of composable arrows
in $\C$
$$\xymatrix@R=24pt{x_0&\ar[l]_(0.4){\textstyle x_{0,1}}x_1&\ar[l]_(0.4){\textstyle x_{1,2}}x_2& \ar[l]\cdots&\ar[l] x_{n-1}&\ar[l]_(0.4){\textstyle x_{n-1,n}}x_n}
$$
and $\ner_0\C=\mbox{Ob}\C$, the set of objects of the category.
By applying realization to $\ner\C$, we obtain the
\emph{classifying space} $\class\C$ of the category \C, that is $$\class\C=|\ner\C|.$$

We also use the notion of classifying space $\class\C$ for a simplicial category ${\C:\Delta^{\!o}\to \cat}$. That is, the realization of the bisimplicial set $\ner\C\!:\!([p],[q])\mapsto\! \ner_q\C_p$   obtained by composing $\C$ with the above functor nerve.
Recall that, when a bisimplicial set
$S\!:\Delta^{\!o}\!\times\!\Delta^{\!o}\to \set$ is regarded as a
simplicial object in the simplicial set category
 and one takes geometric realizations, then one obtains a simplicial space
 $\Delta^{\!o}\to {\bf Top}$, $[p]\mapsto |S_{p,*}|$, whose
Segal realization  is taken to be $|S|$, the geometric realization of
 $S$. As there are natural homeomorphisms \cite[Lemma in p. 86]{quillen}
$$
|[p]\mapsto |S_{p,*}||\cong |\diag S|\cong |[q]\mapsto |S_{*,q}||,
$$ where $\diag S$ is the simplicial set obtained by composing $S$ with the diagonal functor $\Delta\to\Delta\times\Delta$, $[n]\mapsto ([n],[n])$, one usually takes $$|S|=|\diag S|.$$

The following relevant fact is used several times along the
paper (see \cite[Chapter XII, 4.2 and 4.3]{bousfield-kan} or
\cite[IV, Proposition 1.7]{g-j}, for example):

\begin{fact}\label{fact} If
$f\!\!:S\rightarrow S^\prime$ is a bisimplicial map such that the
simplicial maps $ f_{p,*}\!\!: S_{p,*}\rightarrow S^\prime_{p,*}$
(respect. $ f_{*,q}\!\!: S_{*,q}\rightarrow S^\prime_{*,q}$) are
weak homotopy equivalences for all $p$ (respect. $q$), then so is
the map $\diag f\!\!: \!\diag S \rightarrow \!\diag S^\prime$
\end{fact}

A 2-\emph{category} \C\ is just a category enriched in the category of small categories. Then, \C\ is a
category in which the hom-set between any two objects $x,y\in\C$ is the set of objects of a category $\C(x,y)$, whose
arrows are called {\em deformations} (or 2-cells) and are denoted by ${\alpha:u\Rightarrow v}$ and depicted as $$\xymatrix @C=8pt {x  \ar@/^0.7pc/[rr]^{\textstyle u} \ar@/_0.7pc/[rr]_{\textstyle v} & \Downarrow\!\alpha &y .}$$ Composition in each category $\C(x,y),$ usually
referred to as the vertical composition of deformations,  is denoted by juxtaposition. Moreover, the horizontal
composition is a functor $\C(x,y)\times\C(y,z)\overset{\circ\ }\rightarrow \C(x,z)$ that is associative and has
identities $1_x\in\C(x,x)$.

Several times throughout the paper, categories $\C$ are considered as 2-categories in which all deformations are identities, that is, in which each category $\C(x,y)$ is discrete. For any 2-category \C, $\C^{o}$ is the 2-category with the same objects as \C\ but whose hom-categories are $\C^{o}(x,y)=\C(y,x)$.

The {\em nerve of a $2$-category} $\C$ is the simplicial category \begin{equation}\label{nerve}\ner\C: \Delta^{\!o}\to \cat
\end{equation} whose category of $n$-simplices is
$$
\ner_n\C = \bigsqcup_{(x_0,\ldots,x_n)\in \mbox{\scriptsize Ob}\C^{n+1}}\hspace{-0.3cm}
\C(x_1,x_0)\times\C(x_2,x_1)\times\cdots\times\C(x_n,x_{n-1}),
$$
where a typical arrow $\chi$ is a string of deformations in \C
$$
\chi=\xymatrix @C=6pt{x_0  & {\ \Downarrow\! \alpha_{0,1}} & x_1 \ar@/^1pc/[ll]^{\textstyle x_{0,1}}
\ar@/_1pc/[ll]_-{\textstyle x_{0,1}^\prime} & {\ \Downarrow\! \alpha_{1,2}} & x_2  \ar@/^1pc/[ll]^{\textstyle x_{1,2}}
\ar@/_1pc/[ll]_-{\textstyle x_{1,2}^\prime}&\hspace{-8pt}\cdots&\hspace{-10pt}
 x_{n-1}
  & {\ \Downarrow\!\alpha_{n-1, n}} & x_n
\ar@/^1pc/[ll]^{\textstyle x_{n-1,n}} \ar@/_1pc/[ll]_-{\textstyle x_{n-1,n}^\prime }},
$$
and $\ner_0\C=\text{Ob}\C$, as a discrete category. The face and degeneracy functors are defined in a standard way by
using the horizontal composition in $\C$, that is,  $d_0(\chi)= (\alpha_{1,2},\dots,\alpha_{n-1,n})$, $d_1(\chi)=(\alpha_{0,1}\circ
\alpha_{1,2},\dots,\alpha_{n-1,n})$, and so on.

The \emph{classifying space} \class\C\ of the 2-category \C\ is, by definition, the realization of the simplicial
category $\ner\C$, that is, $$\class \C =|\diag\ner\ner\C|,$$ where $\ner\ner\C:([p],[q])\mapsto \ner_q\ner_p\C$ is the
bisimplicial set \emph{double nerve} of the 2-category.

\begin{example}\label{e1}{\em  A strict monoidal category $(\m,\otimes)$ \cite{maclane}, that is, an internal monoid in $\cat$, can be viewed as a
$2$-category with only one object, say $1$,  the objects $x$ of \m\ as morphisms $x:1\rightarrow 1$ and the morphisms of \m\
as deformations. It is the horizontal composition of morphisms and deformations given by the  tensor
${\otimes:\m\times\m\rightarrow\m}$ and the vertical composition of deformations given by the composition of arrows in
$\m$. The identity at the object is the unit object of the monoidal category.

Then,  $\ner(\m,\otimes)$, the nerve of the monoidal category as in (\ref{nerve}), is exactly the simplicial
category that $(\m,\otimes)$ defines by the  reduced bar construction; that is,
$$\ner (\m,\otimes)=\overline{W}(\m,\otimes):\Delta^{\!o}\to \cat,\hspace{0.5cm}[n]\mapsto \m^n.$$
Therefore, the ordinary classifying space $\class(\m,\otimes)$, of the monoidal category, is just the classifying space of the
one object $2$-category it defines. }\qed
\end{example}

A 2-\emph{functor} $F:\B\rightarrow\C$ between 2-categories is an enriched functor and so it takes objects,
morphisms and deformations in \B\ to objects, morphisms and deformations in $\C$ respectively, in such a way that all
the 2-category structure is preserved. It is clear how any 2-functor between 2-categories $F:\B\to\C$ induces a simplicial functor between the corresponding nerves $\ner F:\ner\B \to \ner\C$, whence a cellular map on classifying spaces
$$\class F:\class \B\to \class \C.$$
Furthermore, as we shall explain below, any {\em normal lax functor},  written $$F:\B\rightsquigarrow  \C,$$ also
induces a map $\class F:\class \B\to \class \C$, well defined up to homotopy equivalence. Let us recall that a normal lax functor between $2$-categories, ${F:\B\rightsquigarrow \C}$, consists of three
maps that assign:
\begin{itemize}
\item to each object $x\in \B$, an object $Fx\in\C$;
\item  to each pair of objects $x,y$ of $\B$, a functor $F:\B(x,y)\to\C(Fx,Fy)$;
\item to each pair of composable arrows
$x \xrightarrow{v}y\xrightarrow{u}z$ in $\B$, a deformation  in $\C$ $${F_{u,v}:Fu\circ Fv\Rightarrow
F(u\circ v)}.$$
\end{itemize}
These data are required to satisfy the normalization condition:
$F{1_x}=1_{Fx}$, $F_{u,1}=1_{Fu}=F_{1,u}$; the naturality  condition: deformations $F_{u,v}$ are natural in $(u,v)\in \B(y,z)\times \B(x,y)$; and the coherence (or cocycle) condition that, for any
triple of composable arrows   $x\xrightarrow{w}y\xrightarrow{v}z\xrightarrow{u}t$ in $\B$, the square of
deformations below commutes.
$$
\xymatrix@C=50pt{Fu\circ Fv\circ Fw  \ar@{=>}[d]_{\textstyle F_{{u,v}}\circ
1} \ar@{=>}[r]^{\textstyle 1\circ F_{{v,w}}} &
  Fu\circ F(v\circ w) \ar@{=>}[d]^{\textstyle F_{{u,v\circ w}}} \\
  F(u\circ v)\circ Fw \ar@{=>}[r]^-{\textstyle F_{{u\circ v,w}} }& F(u\circ v\circ w) . }
$$
A normal lax functor in which all constraints $F_{u,v}$ are identities is precisely a 2-functor.

 The {\em geometric nerve} of a 2-category  $\C$, \cite{duskin,street}, is defined to be the simplicial
set \begin{equation}\label{gner}\Delta\C=\Lfunc(-,\C):\Delta^{\!o}\to \set, \end{equation} with $\Delta_n\C=\Lfunc([n],\C)$
the set of normal lax functors $\x: [n]\rightsquigarrow \C$.

 Thus,
for a 2-category  $\C$, the vertices of its geometric nerve $\gner \C$ are  the objects $x_0$ of \C, the 1-simplices
are the morphisms $\xymatrix{x_0&\ar[l]_(0.45){\textstyle x_{0,1}}x_1}$  and the 2-simplices are triangles
$$
\xymatrix@R=20pt@C=20pt{\ar@{}[drr]|(.6)*+{\Downarrow \!{ x_{_{0,1,2}}}\ }  & x_1 \ar[dl]_{\textstyle x_{0,1}}            \\
 x_0  & &  \ar[ll]^{\textstyle x_{0,2}}    x_2\ar[ul]_{\textstyle x_{1,2}} }
$$
with $x_{0,1,2}:x_{0,1}\circ x_{1,2}\Rightarrow x_{0,2}$ a deformation in \C. For $n\geq 3$, an $n$-simplex of $\Delta
\C$ consists of a family
$$
\x=\{x_i,x_{i,j},x_{i,j,k}\}_{0\leq i\leq j\leq k\leq n}
$$
with $x_i$ objects, $x_{i,j}:x_j\rightarrow x_i$ morphisms and $x_{i,j,k}:x_{i,j}\circ x_{j,k}\Rightarrow x_{i,k}$
deformations in \C, which is geometrically represented by a diagram in \C\ with the shape of the 2-skeleton of an
orientated standard $n$-simplex whose faces are triangles
$$
\xymatrix@R=20pt@C=20pt{\ar@{}[drr]|(.6)*+{\Downarrow\! x_{_{i,j,k}}\ }
                & x_j \ar[dl]_{\textstyle x_{i,j}}             \\
 x_i  & &     x_k\,.   \ar[ul]_{\textstyle x_{j,k}} \ar[ll]^{\textstyle x_{i,k}}     }
$$
These data are required to satisfy the condition that each tetrahedron
$$
\begin{array}{ccc}
\xymatrix {
 & x_l \ar[dl]_{\textstyle x_{i,l}}  \ar[dd]|<<<<<<{\textstyle x_{j,l} }\ar[dr]^{\textstyle x_{k,l}}& \\
x_i   & & x_k \ar[dl]^{\textstyle x_{j,k}} \ar[ll]|<<<<<<{\; \textstyle x_{i,k}\;}\\
 & x_j \ar[ul]^{\textstyle
x_{i,j}}}
 & \hspace{1cm}&
\xymatrix@R=2pt{\\  x_{i,j,k}:x_{i,j}\circ x_{j,k}\Rightarrow x_{i,k}\\
x_{i,j,l}:x_{i,j}\circ x_{j,l}\Rightarrow x_{i,l}\\
x_{i,k,l}:x_{i,k}\circ x_{k,l}\Rightarrow x_{i,l}\\
x_{j,k,l}:x_{j,k}\circ x_{k,l}\Rightarrow x_{j,l}}
\end{array}
$$
for $0\leq i\leq j\leq k\leq l \leq n$ is commutative in the sense that the following square of deformations
$$
\xymatrix@C=50pt{x_{i,j}\circ x_{j,k}\circ x_{k,l}  \ar@{=>}[d]_-{\textstyle x_{{i,j,k}}\circ 1}
\ar@{=>}[r]^(0.55){\textstyle 1\circ x_{{j,k,l}}} &
  x_{i,j}\circ x_{j,l} \ar@{=>}[d]^{\textstyle x_{{i,j,l}}} \\
  x_{i,k}\circ x_{k,l} \ar@{=>}[r]^-{\textstyle x_{{i,k,l}}} & x_{i,l}  }
$$
commutes in the category $\C(x_l,x_i)$, and, moreover, the following normalization equations hold:
$$x_{i,i}=1_{x_i}, \ \ x_{i,j,j}=1_{x_{i,j}}=x_{i,i,j}.$$

 Note that, if $\C$ is  a category, regarded as a 2-category with all deformations identities, then
$\gner\C=\ner\C$. As a main result in \cite{b-c}, the following is proved:

\begin{theorem}\label{teo1} For any $2$-category \C\
there is a natural homotopy equivalence
$$
 \class \C\simeq |\gner\C|.
$$
\end{theorem}

\begin{example} \label{e2}{\em Let $(\m,\otimes)$ be a strict monoidal category, regarded as a $2$-category with
 only one object as in Example \ref{e1}.
Then, by Theorem \ref{teo1}, the geometric nerve $\gner(\m,\otimes)$ realizes the classifying space of the monoidal
category, that is, $\class(\m,\otimes)\simeq|\gner(\m,\otimes)|$.
This geometric nerve $\gner(\m,\otimes)$ is a $3$-coskeletal reduced simplicial set whose simplices have the following
simplified interpretation: the $1$-simplices  are the objects of \m, the $2$-simplices are morphisms of $\m$ of the
form
$$
\xymatrix@C=12pt{ x_{0,1}\otimes x_{1,2}\ar[rr]^(0.58){\textstyle x_{0,1,2}} & &  x_{0,2} }
$$
and the $3$-simplices are commutative squares in $\m$ of the form
$$
\xymatrix{x_{0,1}\otimes x_{1,2}\otimes x_{2,3} \ar[rr]^-{\textstyle x_{0,1,2}\otimes 1} \ar[d]_{\textstyle 1\otimes x_{1,2,3}} && x_{0,2}\otimes
x_{2,3}\ar[d]^{\textstyle x_{0,2,3}} \\
x_{0,1}\otimes x_{1,3}\ar[rr]^-{\textstyle x_{0,1,3}} && x_{0,3}.}
$$}
\end{example}

\vspace{0.2cm}
The geometric nerve construction on 2-categories $\C\mapsto\gner\C$ is clearly functorial on normal lax functors
between 2-categories. Therefore, Theorem \ref{teo1} gives the following:

\begin{corollary}
Any normal lax functor between $2$-categories $F:\B\rightsquigarrow\C$ induces a continuous map $\class F:\class\B\to\class
\C$, well defined up to homotopy equivalence.\end{corollary}

The following fact in Lemma \ref{trans} will be needed later. Recall that a {\em lax trasformation}
$\alpha:F\Rightarrow G$, where $F, G:\B\rightsquigarrow\C$ are normal lax functors between  2-categories,
  consists of a family of morphisms $\alpha_x:Fx\to Gx$ in $\C$, one for each object $x$ of $\B$, and deformations
  $\alpha_{u}: \alpha_y\circ Fu \Rightarrow Gu\circ \alpha_x$,
$$
\xymatrix@R=6pt@C=8pt{&Fy\ar[rd]^{\textstyle \alpha_y}&\\ Fx\ar[ru]^{\textstyle Fu}\ar[rd]_{\textstyle \alpha_x}&\hspace{7pt}\Downarrow\!\alpha_u&Gy,
\\&Gx\ar[ru]_{\textstyle Gu}}
$$
which are natural in  $u\in \B(x,y)$, subject to the usual two axioms:  for each object $x $ of $\B$,
$\alpha_{1_x}=1_{\alpha_x}$, and for each pair of composable morphisms $x\overset{v}\to y \overset{u}\to z$ in $\B$,
the diagram below commutes.
$$
\xymatrix@C=-16pt@R=16pt{&&\alpha_z\circ F(u\circ v)\ar@{=>}[rrd]^-{\textstyle \alpha_{u\circ v}}&&\\\alpha_z\circ Fu\circ
Fv\ar@{=>}[rru]^{\textstyle 1\circ F_{u,v}} \ar@{=>}[rd]_(0.4){\textstyle \alpha_u\circ 1}&&&&G(u\circ v)\circ \alpha_x\\&Gu\circ\alpha_y\circ
Fv\ar@{=>}[rr]^(0.46){\textstyle 1\circ \alpha_v}&&Gu\circ G v\circ \alpha_x \ar@{=>}[ru]_(0.6){\textstyle G_{u,v}\circ 1}&}
$$

 Replacing the structure
deformations above by $ \alpha_{u}:Gu\circ \alpha_x\Rightarrow \alpha_y\circ Fu$ we have the notion of {\em  oplax
transformation} $\alpha:F\Rightarrow F'$. If $F,F':\B\to \C$ are $2$-functors, then a lax (or oplax) transformation $\alpha:F\Rightarrow F'$
whose components $\alpha_u$ at the different morphisms $u$ of $\B$ are all identities is called a {\em $2$-natural transformation}. Thus, a $2$-natural transformation
is a \cat-enriched natural transformation, that is, a natural transformation between the underlying ordinary functors that also respects $2$-cells.

\begin{lemma}\label{trans}
If two normal lax functors between $2$-categories, $F,G:\B\rightsquigarrow\C$, are related by a lax or an oplax
transformation, $F\Rightarrow G$, then the induced maps on classifying spaces,
$\class F, \class G:\class\B\to\class\C$, are homotopic.
\end{lemma}
\begin{proof}
Suppose $\alpha:F\Rightarrow G:\B\rightsquigarrow\C$ is a lax transformation. There is a normal lax functor
$H:\B\times [1]\rightsquigarrow \C$ making  the diagram commutative
\begin{equation}\label{homo}
\xymatrix@C=60pt{\B\times [0]\cong \B \ar[d]_{\textstyle 1\times\delta_0}\ar@{~>}[dr]^{\textstyle F}& \\ \B\times[1]\ar@{~>}[r]^{\textstyle H}&\C,\\
\B\times [0]\cong \B\ar@{~>}[ru]_{\textstyle G}\ar[u]^{\textstyle 1\times \delta_1}&}
\end{equation}
that carries a morphism in $\B\times[1]$ of the form $(x\overset{u}\to y,1\to 0):(x,1)\to (y,0)$ to the composite
morphism in $\C$
$$Fx\overset{ \alpha_x}\longrightarrow Gx\overset{  Gu}\longrightarrow  Gy,$$
and a deformation $(\phi,1_{1\to 0}):(x\overset{u}\to y,1\to 0)\Rightarrow (x\overset{v}\to y,1\to 0)$ to
$$ G\phi\circ 1_{\alpha_x}:
 Gu\circ \alpha_x\Rightarrow  Gv\circ\alpha_x.$$

For $x\overset{v}\to y\overset{u}\to z$ two composable morphisms in $\B$,  the structure constraints
$$H(y\overset{u}\to z,1\to 0)\circ H(x\overset{v}\to y,1\to 1)\Rightarrow H(x\overset{u\circ v}\to z,1\to 0)$$
and
$$H(y\overset{u}\to z,0\to 0)\circ H(x\overset{v}\to y,1\to 0)\Rightarrow H(x\overset{u\circ v}\to z,1\to 0)$$
are, respectively, given by the composite deformations
$$
\xymatrix{ Gu\circ\alpha_y\circ Fv\ar@{=>}[r]^(0.47){1\circ
\alpha_v}& Gu\circ  Gv\circ \alpha_x \ar@{=>}[r]^{ G_{u,v}\circ
1}& G(u\circ v)\circ \alpha_x}
$$
and
$$
\xymatrix{ Gu\circ  Gv \circ \alpha_x\ar@{=>}[r]^{ G_{u,v}\circ
1}&  G(u\circ v)\circ \alpha_x.}
$$

Applying geometric nerve construction  to diagram (\ref{homo}), we obtain a diagram of simplicial set maps
$$
\xymatrix@C=60pt{\gner\B\times \gner[0]\cong \gner\B \ar[d]_{\textstyle 1\times \delta_0}\ar[dr]^{\textstyle \gner F}& \\
\gner\B\times\gner[1]\ar[r]^{\textstyle  \gner H}&\gner\C,\\
\gner\B\times \gner[0]\cong \gner\B\ar[ru]_{\textstyle \gner G}\ar[u]^{\textstyle 1\times\delta_1}&}
$$
showing that the simplicial maps $\gner F,\gner G:\gner\B\to \gner\C$ are made homotopic by $\gner H$, whence the lemma follows by
Theorem \ref{teo1}.

The proof is similar for the case in which $\alpha:F\Rightarrow G:\B\rightsquigarrow\C$  is an oplax transformation,
but with a change in the construction of the lax functor $H:\B\times[1]\rightsquigarrow \C$ that makes diagram (\ref{homo}) commutative: now define  $H$ such that $H(x\overset{u}\to y, 1\to 0)=\alpha_y\circ Fu:Fx\to Gy.$ \qed
\end{proof}

\section{Homotopy cartesian squares induced by 2-functors}

Suppose $F:\B\to\C$ any given 2-functor between 2-categories \B\ and \C. For each object $z$ of $\C$, by the
\emph{homotopy fibre} 2-category $$\comma {z} F$$ we mean Gray's lax comma category $\comma
{\ulcorner z\urcorner} F$ where $\ulcorner z \urcorner:\mathbf{1}\rightarrow
\C$ is the ``name of an object" $z$ 2-functor \cite[\S 3.1]{gray}. Its objects are
the pairs $(x,v)$, where $x$ is an object of \B\ and $v:z\to Fx$ is a morphism in \C. A morphism $(u,\beta):(x,v)\to
(x',v')$ consists of a morphism $u\!:\! x\to x'$ in \B\ together with a deformation $\beta\!:\!Fu\circ v\Rightarrow \!v'$ in $\C$
$$
\xymatrix@R=15pt@C=15pt{\ar@{}[drr]|(.55)*+{\textstyle \Downarrow\!\beta}
                & Fx \ar[dl]_{\textstyle Fu}             \\
 Fx'  & &     z \,,  \ar[ul]_{\textstyle v} \ar[ll]^{\textstyle v'}     }
$$
and a deformation
$$\xymatrix@C=10pt{(x,v)  \ar@/^1pc/[rr]^{\textstyle (u,\beta)} \ar@/_1pc/[rr]_{\textstyle (u',\beta')} & {\ \Downarrow \!\alpha} &(x',v') }  $$
is a deformation $\alpha:u\Rightarrow u'$
in \B\ such that the diagram in $\C(z,Fx')$
$$
\xymatrix@C=10pt@R=15pt{Fu\circ v\ar@{=>}[rr]^{\textstyle \beta}\ar@{=>}[rd]_{\textstyle F\!\alpha\circ 1_v}&&v'\\&Fu'\circ v \ar@{=>}[ru]_{\textstyle \beta'}&}
$$
commutes. Compositions in $\comma {z} F$ derive naturally from those in $\B$ and $\C$.

Similarly, one has the homotopy fibre 2-categories $\comma{F}z$, with objects the pairs $(x,Fx\overset{v}\to z)$. The
arrows $(u,\beta):(x,v)\to (x',v')$, are pairs where $u: x\to x'$ is a morphism in $\B$ and $\beta:v'\circ
Fu \Rightarrow v$ is a deformation in $\C$; and a deformation $\alpha:(u,\beta)\Rightarrow (u',\beta'):(x,v)\to (x',v')$ is a
deformation $\alpha:u\Rightarrow u'$ in $\B$ such that $\beta'(1_{v'}\circ F\alpha)=\beta$.

In particular, when $F=1_\C$ is the identity 2-functor on \C\, we have the {\em comma $2$-categories} $\comma{z}\C$, of objects
under an object $z$, and $\comma{\C}z$, of objects over $z$. The following fact is proved in \cite[Proposition 4.1]{b-c}:

\begin{lemma}\label{cont}
For any object $z$ of a $2$-category \C, the classifying spaces of the comma $2$-categories $\class(\comma{z}\C)$ and $\class(\comma{\C}z)$ are contractible.
\end{lemma}
\begin{proof} We shall discuss the case of $\class(\comma{z}\C)$. Let $\mbox{Ct}_z:\comma{z}\C\to \comma{z}\C$ be the constant 2-functor sending every object
of $\comma{z}\C$ to $(z,1_z)$, every morphism to the identity morphism $1_{(z,1_z)}=(1_z,1_{1_z})$ and every
deformation to the identity deformation $1_{1_z}:1_{(z,1_z)}\Rightarrow 1_{(z,1_z)}$. There is a canonical oplax
transformation $\mbox{Ct}_z\Rightarrow 1_{\comma{z}\C}$ whose component at an object $(x,z\overset{v}\to x)$ is the
morphism $(v,1_v):(z,1_z)\to (x,v)$, and whose component at a morphism $(u,\beta):(x,v)\to (x',v')$ is $\beta$. From
Lemma \ref{trans}, it follows that $\class (1_{\comma{z}\C})= 1_{\class (\comma{z}\C)}$ and $\class \mbox{Ct}_z=
\mbox{Ct}_{\class z}$ are homotopic, whence the lemma. \qed\end{proof}

Returning to an arbitrary 2-functor $F:\B\to\C$, observe that, for any object $z$ of $\C$, there is a $2$-functor
$$
\Phi:\comma{z}{F}\to \B,$$ defined by forgetting the second components $$ \xymatrix @C=8pt{(x,v) \ar@/^1pc/[rr]^{\textstyle (u,\beta)} \ar@/_1pc/[rr]_-{\textstyle (u',\beta')}
& {\Downarrow\!\alpha} & (x',v')} \xymatrix{\ar@{|->}[r]^{\textstyle \Phi}&} \xymatrix @C=5pt{x
\ar@/^0.7pc/[rr]^{\textstyle u} \ar@/_0.7pc/[rr]_-{\textstyle u'} & {\Downarrow\!\alpha} &
x'},
$$
and there is a commutative square of $2$-functors
\begin{equation}\label{sq1}
\xymatrix{\comma{z}F\ar[r]^{\textstyle \Phi}\ar[d]_{\textstyle F'}& \B\ar[d]^{\textstyle F}&\\
\comma{z}\C\ar[r]^{\textstyle \Phi}&\C,&}
\end{equation}
where $F':\comma{z}F\to\comma{z}\C$ is given by
$$
 \xymatrix @C=8pt{(x,v) \ar@/^1pc/[rr]^{\textstyle (u,\beta)} \ar@/_1pc/[rr]_-{\textstyle (u',\beta')}
& {\Downarrow\!\alpha} & (x',v')} \xymatrix{\ar@{|->}[r]^{\textstyle F'}&} \xymatrix @C=8pt{(Fx,v)\
\ar@/^1pc/[rr]^{\textstyle (Fu,\beta)} \ar@/_1pc/[rr]_-{\textstyle (Fu',\beta')} & {\Downarrow\!F\alpha} &
(Fx',v')}.
$$

Moreover, for any morphism $w: z_1\to z_0$ in \C\,  the assignment
$$
 \xymatrix @C=8pt{(x,v) \ar@/^1pc/[rr]^{\textstyle (u,\beta)} \ar@/_1pc/[rr]_-{\textstyle (u',\beta')}
& {\Downarrow\! \alpha} & (x',v')} \xymatrix{\ar@{|->}[r]^{\textstyle w^*}&} \xymatrix @C=1pc{(x,v\circ w)\
\ar@/^1pc/[rr]^{\textstyle (u,\beta\circ 1_w)} \ar@/_1pc/[rr]_-{\textstyle (u',\beta'\circ 1_w)} & {\Downarrow\! \alpha} &
(x',v'\circ w)},
$$
defines a 2-functor $w^*: \comma{z_0}F \to \comma{z_1}F$, and the analogous to Quillen's Theorem B for 2-functors is stated as follows:

\begin{theorem}\label{B} Let $F:\B\to\C$ be a $2$-functor such that, for every morphism $z_1\to z_0$ in $\C$, the induced
map $\class(\comma{z_0}F)\to \class(\comma{z_1}F)$ is a homotopy
equivalence. Then, for every object $z$ of $\C$, the induced square
by $(\ref{sq1})$ on classifying spaces
\begin{equation}\label{hc}
\xymatrix{\class(\comma{z}F)\ar[r]\ar[d]&\class \B\ar[d]\\
\class(\comma{z}\C)\ar[r]&\class\C}
\end{equation}
is homotopy cartesian.  Therefore, for each object $x\in F^{-1}z$,
there is a homotopy fibre sequence $$\class (\comma{z}F)\to \class
\B\to\class\C,$$ relative to the base points $z$ of $\class \C$, $x$
of $\class \B$ and $(x,1_z)$ of $\class(\comma{z}F)$, that induces a
long exact sequence on homotopy groups
$$\cdots \to \pi_{n+1}\class\C\to\pi_n\class(\comma{z}F)\to\pi_n\class\B\to\pi_n\class\C\to\cdots.$$
\end{theorem}

The remainder of this section is mainly devoted to the proof of this theorem. To do so, we shall start with the following construction of
2-categories $\comma{[q]}F$, one for each integer $q\geq 0$,  which,
in the case of $q=0$ and up to almost obvious identification,  yields
the coproduct $\bigsqcup_{z}\comma{z}F$, of the different homotopy
fibre 2-categories:

For each $q\geq 0$, let $$\comma{[q]}F$$ be the 2-category whose
objects are pairs $(x,\mathbf{v})$, where $x$ is an object of $\B$
and $\mathbf{v}:[q+\!1]\!\rightsquigarrow \C$ is a normal lax
functor such that $v_0=Fx$. A morphism $(u,\y)$ consists of a morphism $u:x\to x'$ in
$\B$ together with a normal lax functor $\y:[q+\!2]\rightsquigarrow \C$
such that $y_{0,1}=Fu$; the source of $(u,\y)$ is $(x,\y\delta_0)$ and its
target is $(x',\y\delta_1)$. Note that, since the square
$$
\xymatrix@R=16pt@C=16pt{[q]\ar[r]^{\delta_0}\ar[d]_{\delta_0}&[q+\!1]\ar[d]^{\delta_0}\\[q+\!1]\ar[r]^{\delta_1}&[q+\!2]}
$$
is commutative and cartesian in $\mathbf{Cat}$, given two objects
$(x,\mathbf{v})$ and $(x',\mathbf{v'})$ of $\comma{[q]}F$, the
existence of a morphism between them requires that
$\mathbf{v}\delta_0=\mathbf{v'}\delta_0$, and then such a morphism
$(u,\y):(x,\mathbf{v})\to (x',\mathbf{v'})$ is completely specified
by the morphism $u:x\to x'$ and the deformations
$$
\xymatrix{\ar@{}[drr]|(.6)*+{\hspace{-10pt}\Downarrow\!
y_{{0,1,i\text{+}2}}}
                & Fx \ar[dl]_{\textstyle Fu}             \\
 Fx'  & &     v_{i\text{+}1}\!=\!v'_{i\text{+}1}   \ar[ul]_{\textstyle v_{0,i\text{+}1}} \ar[ll]^{\textstyle v'_{0,i\text{+}1}}     }
$$
for $0\leq i\leq q$. A deformation in  $\comma{[q]}F$
$$\xymatrix@C=10pt@R=15pt{(x,\lv)  \ar@/^1pc/[rr]^{\textstyle (u,\y)} \ar@/_1pc/[rr]_{\textstyle (u',\y')} & {\Downarrow \alpha} &(x',\lv') }
  $$
is a deformation $\alpha:u\Rightarrow u'$ in \B\ such that the
triangles
$$
\xymatrix@C=10pt@R=15pt{Fu\circ
v_{0,i\text{+}1}\ar@{=>}[rr]^{\textstyle
y_{{0,1,i\text{+}2}}}\ar@{=>}[rd]\ar@{}@<-3pt>[rd]_(0.4){\textstyle
F\!\alpha\circ 1}&&v'_{0,i\text{+}1}\\&Fu'\circ v_{0,i\text{+}1}
\ar@{=>}[ru]_(0.6){\textstyle y'_{{0,1,i\text{+}2}}}&}
$$
commute, for  $0\leq i\leq q$.

Compositions in $\comma {[q]} F$ come from those in \B\ and \C\ in
the natural way. Thus,  a morphism
$(u,\y):(x,\lv)\to (x',\lv')$ composes horizontally with a morphism $(u',\y'):(x',\lv')\to
(x'',\lv'')$ yielding the morphism $(u'\circ u,\y''):(x,\lv)\to
(x'',\lv'')$, where, for $0\leq i\leq q$, the deformation
$y''_{{0,1,i\text{+}2}}: F(u'\circ u)\circ v_{0,i\text{+}1}
\Rightarrow v''_{0,i\text{+}1}$ is the composition
$$
\xymatrix@C=50pt{Fu'\circ Fu\circ
v_{0,i\text{+}1}\ar@{=>}[r]^-{\textstyle 1\circ
y_{{0,1,i\text{+}2}}}&Fu'\circ
v'_{0,i\text{+}1}\ar@{=>}[r]^-{\textstyle
y'_{{0,1,i\text{+}2}}}&v''_{0,i\text{+}1}.}
$$
The identity morphism for each object in $\comma {[q]} F$ is
provided by the surjection $\sigma_0:[q+1]\to [q]$ which repeats the
$0$ element, by $ 1_{(x,\lv)}=(x,\lv\sigma_0)$. And deformations in
$\comma {[q]} F$ compose, both horizontally and vertically, as in
$\B$.

Similarly, one has the 2-categories
$$
\comma{F}[q],\hspace{0.6cm} q\geq 0,
$$
with objects the pairs  $(x,\mathbf{v})$, where $x$ is an object of $\B$
and $\mathbf{v}:[q+\!1]\!\rightsquigarrow \C$ is a normal lax
functor such that $v_{q+1}=Fx$. The morphisms $(u,\y):(x,\mathbf{v})\to (x',\mathbf{v'})$, are pairs where $u:x\to x'$ is a morphism in $\B$ and $\y:[q+2]\rightsquigarrow \C$ is a normal lax functor such that $y_{q+1,q+2}=Fu$, $\y\delta_{q+1}=\lv$ and $\y\delta_{q+2}=\lv'$; and a deformation $\alpha:(u,\y)\Rightarrow (u',\y'):(x,\lv)\to (x',\lv')$ is a deformation $\alpha:u\Rightarrow u'$ in $\B$ such that $y'_{i,q+1,q+2}(1_{v'_{i\!,\!q+1}}\!\circ F\alpha)=y_{i,q+1,q+2}$, for $0\leq i\leq q$.

\vspace{0.2cm}

Regarding the set $\gner_q\C$, of geometric $q$-simplices of the
2-category $\C$, as a discrete  2-category whose morphisms and
deformations are all identities, the 2-functors
\begin{equation}\label{npsi}
\Psi_{\!q}:\comma{[q]}F\longrightarrow \Delta_q\C, \hspace{0.3cm} \Psi'_{\!q}:\comma{F}[q]\longrightarrow \Delta_q\C
\end{equation}
are respectively defined by
$$(x,\lv)\overset{\textstyle \Psi}\longmapsto \lv\delta_0,\hspace{0.3cm}
(x,\lv)\overset{\textstyle \Psi'}\longmapsto \lv\delta_{q+1}.$$
And these 2-functors give rise to a decomposition of the 2-categories
$\comma{[q]}F$ and $\comma{F}[q]$ as
$$
\comma{[q]}F= \bigsqcup_{\z:[q]\rightsquigarrow\C}\comma{\z}F,\hspace{0.3cm}
\comma{F}[q]= \bigsqcup_{\z:[q]\rightsquigarrow\C}\comma{F}\z,
$$
where, for each geometric q-simplex $\z:[q]\rightsquigarrow\C$, the
{\em homotopy fibre $2$-categories of $F$ at $\z$}
$$\comma{\z}F:=\Psi^{-1}_{\!q}(\z),\hspace{0.3cm}
\comma{F}\z:=\Psi'^{-1}_{\!q}(\z),$$ are respectively  defined to be the full
2-subcategories of $\comma{[q]}F$ and $\comma{F}[q]$ fibre of $\Psi_{\!q}$ and $\Psi'_{\!q}$ at $\z$.

In particular, when $F=1_\C$ is the identity 2-functor on \C\, we have the comma 2-categories $$\comma{\z}\C,\hspace{0.3cm}  \comma{\C}\z,$$ of objects
under and over a normal lax functor $\z:[q]\rightsquigarrow \C$.

\vspace{0.2cm}
The following required lemma does not use any hypothesis on the 2-functor $F$ in Theorem (\ref{B}).

\begin{lemma}\label{lemma2}
Let  $\z\!:\![q]\rightsquigarrow \C$ be any given normal lax functor. Then,
\begin{enumerate}
\item[(i)] There are $2$-functors
\begin{equation}\label{gamz} \Gamma:\comma{z_0}F\to \comma{\z}F,\hspace{0.3cm} \Gamma':\comma{F}z_q\to\comma{F}\z,
\end{equation} both
inducing
homotopy equivalences on classifying spaces $$\class(
\comma{z_0}F)\simeq  \class(\comma{\z}F),\hspace{0.3cm} \class(\comma{F}z_q)\simeq \class(\comma{F}\z).$$
\item[(ii)] The classifying spaces of the comma $2$-categories $\class(\comma{\z}\C)$ and $\class(\comma{\C}\z)$ are contractible.
    \end{enumerate}
\end{lemma}
\proof $(i)$  We shall discuss the case of $\Gamma$, which is defined as follows: It carries an object $(x,v)$ of $\comma{z_0}F$ to the object $(x,v^\z)$ of $\comma{\z}F$, where
 $v^\z:[q+1]
 \rightsquigarrow \C$ is the normal lax functor defined by the equalities
 $$\begin{array}{ll}v^\z\delta_0=\z,\hspace{0.3cm}v^\z_0=Fx,\\[4pt] v^\z_{0,i+1}=v\circ z_{0,i}:z_i\overset{z_{0,i}}\to
 z_0\overset{v}\to Fx,\\[4pt]
v^\z_{0,i+1,j+1}=1_v\circ z_{0,i,j}:v\circ z_{0,i}\circ
z_{i,j}\Rightarrow v\circ z_{0,j}.\end{array}$$  A morphism $(u,y):(x,v)\to (x',v')$ in $\comma{z_0}F$ is
carried by $\Gamma$ to the morphism $(u,y^\z):(x,v^\z)\to
(x',v'^\z)$ of $\comma{\z}F$ specified by the deformations
$$y^\z_{0,1,i+2}=y\circ 1_{z_{0,i}}:Fu\circ v\circ
z_{0,i}\Longrightarrow v'\circ z_{0,i},$$ for $0\leq i\leq q$, and
on deformations, the 2-functor acts by the simple rule
$\Gamma(\alpha)\!=\!\alpha$.

Actually, this 2-functor $\Gamma$ embeds $\comma{z_0}F$ into
$\comma{\z}F$ as a deformation retract,
 with retraction given by the 2-functor
$$\Theta:\comma{\z}F\to \comma{z_0}F,$$
$$
\xymatrix@C=10pt@R=15pt{(x,\lv)  \ar@/^1pc/[rr]^{\textstyle (u,\y)}
\ar@/_1pc/[rr]_{\textstyle (u',\y')} & {\Downarrow \alpha}
&(x',\lv') }\overset{\textstyle \Theta}\longmapsto
\xymatrix@C=10pt@R=15pt{(x,v_{0,1})  \ar@/^1pc/[rr]^{\textstyle
(u,y_{0,1,2})} \ar@/_1pc/[rr]_{\textstyle (u',y'_{0,1,2})} &
{\Downarrow \alpha} &(x',v'_{0,1}). }
$$

One observes that $\Theta \Gamma =1_{\comma{z_0}F}$. Furthermore,
there is a 2-natural transformation $\Gamma\Theta\Rightarrow
1_{\comma{\z}F}$, whose component at an object $(x,\lv)$ of
$\comma{\z}F$ is the morphism $(1_x,\tilde{\lv}):(x,v_{0,1}^\z)\to
(x,\lv)$ with
$$\tilde{\lv}_{0,1,i+2}=v_{0,1,i+1}:v_{0,1}\circ z_{0,i}\Longrightarrow v_{0,i+1},$$
for $0\leq i\leq q$. Then, by Lemma \ref{trans}, it follows that the
induced map $\class\Gamma$ embeds the space $\class (\comma{z_0}F)$
into $\class(\comma{\z}F)$ as a deformation retract, with
$\class\Theta$ a retraction.

The proof for $\Gamma'$ is similar. It carries an object $(x,v)$ of $\comma{F}z_p$ to the object $(x,\z^v)$ of $\comma{F}\z$, where $\z^v_{i,q+1}=z_{i,q}\circ v$, etc.

\vspace{0.2cm}
$(ii)$  After what has already been proven above, the result follows from Lemma \ref{cont}. \qed

\vspace{0.3cm}

Actually, we have a simplicial 2-category
\begin{equation}\label{ns}\comma{[-]}F=\bigsqcup_{q\geq 0}\comma{[q]}F,
\end{equation}
in which the induced 2-functor $\xi^*:\comma{[n]}F\to
\comma{[q]}F$, by a map $\xi:[q]\to [n]$ in the simplicial category,
is given by
$$
\xymatrix@C=10pt@R=15pt{(x,\lv)  \ar@/^1pc/[rr]^{\textstyle (u,\y)}
\ar@/_1pc/[rr]_{\textstyle (u',\y')} & {\Downarrow \alpha}
&(x',\lv') }\overset{\textstyle \xi^*}\longmapsto
\xymatrix@C=10pt@R=15pt{(x,\lv(\xi\text{+}1))
\ar@/^1pc/[rr]^{\textstyle (u,\y(\xi\text{+}2))}
\ar@/_1pc/[rr]_{\textstyle (u',\y'(\xi\text{+}2))} & {\Downarrow
\alpha} &(x',\lv'(\xi\text{+}1)), }
$$
where, for any integer $p\geq 0$, the map $\xi+p:[q+p]\to[n+p]$ is
$$
(\xi+p)(i)=\left\{\begin{array}{lll}i&\text{if}&0\leq i<p,\\[4pt]\xi(i-p)+p&\text{if}&p\leq i\leq q+p.\end{array}\right.
$$

Note that, for any $\xi:[q]\to [n]$, the square below in the simplicial category commutes.
$$
\xymatrix@R=16pt@C=24pt{[q]\ar[r]^{\textstyle \xi}\ar[d]_{\textstyle
\delta_0}&[n]\ar[d]^{\textstyle \delta_0}
\\ [q+\!1]\ar[r]^{\textstyle \xi\!+\!1}&[n+1]}
$$
Hence, for any $\z:[n]\rightsquigarrow \C$, the 2-functor $\xi^*:\comma{[n]}F\to
\comma{[q]}F$ maps the homotopy fibre 2-category $\comma{\z}F$ into $\comma{\z\xi}F$. As a crucial result in our discussion for proving Theorem \ref{B}, we have the following:

\begin{lemma}\label{lemma22} Under the hypothesis in Theorem $\ref{B}$, for any given normal lax functor  ${\z:[n]\rightsquigarrow \C}$ and
map  in the simplicial category $\xi:[q]\to [n]$,  the $2$-functor
$$\xi^*:\comma{\z}F\to \comma{\z\xi}F$$ induces a homotopy
equivalence on classifying spaces
$\class(\comma{\z}F)\simeq \class(\comma{\z\xi}F).$
\end{lemma}
\proof We have the square of 2-functors
$$
\xymatrix{\comma{z_0}F\ar[r]^{\textstyle \Gamma_\z}\ar[d]_{\textstyle z_{0,\xi(0)}^*}&
\comma{\z}F\ar[d]^{\textstyle \xi^*}\\
\comma{z_{\xi(0)}}F\ar[r]^{\textstyle
\Gamma_{\z\xi}}&\comma{\z\xi}F,}
$$
where the $\Gamma's$ are those 2-functors in Lemma \ref{lemma2} (i) corresponding to $\z$ and $\z\xi$ respectively. The two composite 2-functors in the square are related by a 2-natural transformation
$\Gamma_{\z\xi}\,z_{0,\xi(0)}^*\Rightarrow \xi^*\,\Gamma_{\z}$,
whose component at an object $(x,v)$ of $\comma{z_0}F$ is the
morphism $$(1_x,\y):(x,(v\circ z_{0,\xi(0)})^{\z\xi})\to
(x,v^\z(\xi+1))$$ in $\comma{\z\xi}F$ with
$$
y_{0,1,i+2}=1_v\circ z_{0,\xi(0),\xi(i)}:v\circ z_{0,\xi(0)}\circ
z_{\xi(0),\xi(i)}\Longrightarrow v\circ z_{0,\xi(i)},
$$
for $0\leq i\leq q$. Hence, by Lemma \ref{trans}, the induced square
on classifying spaces
$$
\xymatrix{\class(\comma{z_0}F)\ar[r]^{\textstyle \class\Gamma_\z}\ar[d]_{\textstyle \class z_{0,\xi(0)}^*}&
\class(\comma{\z}F)\ar[d]^{\textstyle \class \xi^*}\\
\class(\comma{z_{\xi(0)}}F)\ar[r]^{\textstyle
\class\Gamma_{\z\xi}}&\class(\comma{\z\xi}F)}
$$
is homotopy commutative; that is, there is a homotopy
$\class\Gamma_{\z\xi}\,\class z_{0,\xi(0)}^*\simeq
\class\xi^*\,\class\Gamma_{\z}$. Since both maps $\class\Gamma_\z$
and $\class\Gamma_{\z\xi}$ are homotopy equivalences by Lemma
\ref{lemma2} above, and, by hypothesis,  $\class z_{0,\xi(0)}^*$ is
also a homotopy equivalence, the result follows. \qed

The next lemma plays, in our discussion,  the role of the relevant
Lemma \cite[p. 90]{quillen} in Quillen's proof of his Theorem B
for functors. The simplicial 2-category (\ref{ns}) has a classifying
space $\class(\comma{[-]}F)$; namely, the geometric realization of
the simplicial space $[q]\mapsto \class(\comma{[q]}F)$ obtained by
replacing each 2-category $\comma{[q]}F$ by its classifying space,
or
$$\xymatrix{\class(\comma{[-]}F)=|\diag\gner(\comma{[-]}F)|},$$
where $\gner(\comma{[-]}F):([p],[q])\mapsto \gner_p(\comma{[q]}F)$
is the bisimplicial set obtained by applying the geometric nerve
functor to the simplicial 2-category $\comma{[-]}F$. Furthermore,
regarding the simplicial set $\gner\C$, geometric nerve of the
2-category $\C$, as a  simplicial discrete 2-category whose
morphisms and deformations are all identities, a simplicial
2-functor
\begin{equation}\label{2psi}
\Psi:\comma{[-]}F\longrightarrow \Delta\C, \hspace{0.6cm}(x,\lv)\to
\lv\delta_0,
\end{equation}
is defined by the 2-functors (\ref{npsi}). This 2-functor yields an
induced map on classifying spaces $\class(\comma{[-]}F)\to
\class\C$, and we have the following:

\begin{lemma}\label{car}
For any object $z_0$ of $\C$, the induced square by the inclusion of
$\comma{z_0}F$ into $\comma{[-]}F$ and the simplicial $2$-functor {\em
(\ref{2psi})}
\begin{equation}\label{cars}
\xymatrix{\class(\comma{z_0}F)\ar[r]\ar[d]&\class(\comma{[-]}F)\ar[d]^{\textstyle \class\Psi}\\
pt\ar[r]^{\textstyle z_0}&\class\C,}
\end{equation}
where  $pt$ is the one-point space, is  homotopy cartesian. That is,
$\class(\comma{z_0}F)$ is  homotopy equivalent to the homotopy fibre
of $\class\Psi$, relative to the base point $z_0$ of $\class\C$.
\end{lemma}

\proof  By Theorem \ref{teo1}, it suffices to prove that the diagram
of bisimplicial sets
$$
\xymatrix{\Delta(\comma{z_0}F)\ar[r]\ar[d]&\gner(\comma{[-]}F)\ar[d]^{\textstyle \Delta\Psi}\\
\Delta[0]\ar[r]^{\textstyle z_0}&\Delta\C}$$ induces on
diagonals a homotopy cartesian square of simplicial maps. And, to do
that, we are going to prove the following two facts:
\begin{enumerate}
\item[(i)] The pullback square of bisimplicial sets
$$\xymatrix{\Delta[0]\times_{\Delta\C}\gner(\comma{[-]}F)\ar[r]
\ar[d]&\gner(\comma{[-]}F)
\frac{}{}\ar[d]^{\textstyle \Delta\Psi}\\
\Delta[0]\ar[r]^{\textstyle z_0}&\Delta\C,}
$$
induces on diagonals a homotopy cartesian square of simplicial maps.

\item[(ii)] The induced map on diagonals by the bisimplicial map
\begin{equation}\label{corner}\xymatrix{\Delta(\comma{z_0}F)\to \Delta[0]\times_{\Delta\C}
\gner(\comma{[-]}F)}\end{equation} is a weak homotopy equivalence.
\end{enumerate}

Recall that the homotopy fibre of $\diag\Delta\Psi$ at $z_0$ is the
pullback
$$\xymatrix{Y\times_{\gner \C}\diag \gner(\comma{[-]}F),}$$ where
$\gner[0]\to Y\to\gner\C$ is a (any) factorization of
$\gner[0]\overset{z_0}\to \gner\C$ into a trivial cofibration (=
injective weak equivalence) followed by a Kan fibration. By using
the ``small object argument" \cite{quillen67,g-j} to find  such a
factorization, the trivial cofibration $\gner[0]\to Y$ is a colimit
of a sequence of pushouts of coproducts of simplicial inclusions
$\Lambda^{\!k}[n]\hookrightarrow \gner[n]$, of $k^{\text{th}}$-horn
subcomplexes of standard $n$-simplex simplicial complexes. Since
pullbacks along $\diag \Delta\Psi$ commute with colimits in the
comma category of simplicial sets over $\gner \C$, to prove (i) it
is sufficient to show that for each composite
$$
\Lambda^{\!k}[n]\hookrightarrow \gner[n]\overset{\z}\to \gner\C,
$$
with $z_0$ the given, the bisimplicial map obtained by pullingback
along $\Delta\Psi$
\begin{equation}\label{pub}
\xymatrix{\Lambda^{\!k}[n]\times_{\gner \C}\gner(\comma{[-]}F)
\longrightarrow \gner[n]\times_{\gner \C}\gner(\comma{[-]}F)}
\end{equation}
induces a weak homotopy equivalence on diagonals.

Fix a simplicial map $\z:\gner[n]\to  \gner\C$, that is, a normal
lax functor $\z:[n]\rightsquigarrow \C$.

The pullback square of bisimplicial sets
$$\xymatrix{\Delta[n]\times_{\Delta\C}\gner(\comma{[-]}F)\ar[r]
\ar[d]&\gner(\comma{[-]}F)
\frac{}{}\ar[d]^{\textstyle \Delta\Psi}\\
\Delta[n]\ar[r]^{\textstyle \z}&\Delta\C,}
$$
is that induced on geometric nerves by the pullback of simplicial
2-categories
$$\xymatrix{\Delta[n]\times_{\Delta\C}\comma{[-]}F\ar[r]
\ar[d]&\comma{[-]}F
\frac{}{}\ar[d]^{\textstyle \Psi}\\
\Delta[n]\ar[r]^{\textstyle \z}&\Delta\C,}
$$
where
$$\Delta[n]\times_{\Delta\C}\comma{[-]}F=\bigsqcup_{q}\bigsqcup_{[q]\overset{\xi}\to
[n]}\comma{\z\xi}F,
$$
is the simplicial 2-subcategory of $\comma{[-]}F$ generated by the
2-subcategory $\comma{\z}F$ of $\comma{[n]}F$.

On considering the product simplicial 2-category
$$
\Delta[n]\times
\comma{z_0}F=\bigsqcup_{q}\bigsqcup_{[q]\overset{\xi}\to
[n]}\comma{z_0}F,
$$
we have a simplicial 2-functor
\begin{equation}\label{gamma}
\gamma:\Delta[n]\times \comma{z_0}F\longrightarrow
\Delta[n]\times_{\Delta\C}\comma{[-]}F,
\end{equation}
$$\gamma=\bigsqcup_q\bigsqcup_{[q]\overset{\xi}\to [n]}\Big(\xymatrix@C=14pt{\comma{z_0}F
\ar[r]^{\Gamma}&\comma{\z}F\ar[r]^{\xi^*}&\comma{\z\xi}F}\Big),$$
where $\Gamma$ is the 2-functor (\ref{gamz}). Since, for any $q\geq
0$, the map induced by $\gamma_q$ on geometric nerves
$$
\bigsqcup_{[q]\overset{\xi}\to
[n]}\Big(\xymatrix@C=14pt{\gner(\comma{z_0}F)
\ar[r]^{\Delta\Gamma}&\gner(\comma{\z}F)\ar[r]^{\Delta\xi^*}&\gner(\comma{\z\xi}F)}\Big)
$$
is a coproduct of weak homotopy equivalences, by Lemmas \ref{lemma2}
and \ref{lemma22}, it follows that  the  simplicial map induced on
diagonals
\begin{equation}\label{dgamma}
\Delta\gamma:\gner[n]\times \gner(\comma{z_0}F)\longrightarrow
\diag\left(\gner[n]\times_{\gner \C}\Delta(\comma{[-]}F)\right)
\end{equation}
is a weak homotopy equivalence as well, by Fact \ref{fact}.

Similarly, the bisimplicial set $\Lambda^{\!k}[n]\times_{\gner
\C}\gner(\comma{[-]}F)$ is the geometric nerve of the simplicial
2-category
$$\Lambda^{\!k}[n]\times_{\Delta\C}\comma{[-]}F=\bigsqcup_{q\geq
0}\hspace{-0.3cm}\bigsqcup_{\scriptsize
\begin{array}{c}[q]\overset{\xi}\to [n]\\ \exists j\neq k~|~ \beta \text{ misses }
j\end{array}}\hspace{-0.6cm}\comma{\z\xi}F
$$
and, by taking into account the product simplicial 2-category
$$
\Lambda^{\!k}[n]\times\comma{z_0}F=\bigsqcup_{q\geq
0}\hspace{-0.3cm}\bigsqcup_{\scriptsize
\begin{array}{c}[q]\overset{\xi}\to [n]\\ \exists j\neq k~|~ \beta \text{ misses }
j\end{array}}\hspace{-0.6cm}\comma{z_0}F,
$$
the simplicial 2-functor $\gamma$ above (\ref{gamma}) restricts to
a simplicial 2-functor
$$\gamma':\Lambda^{\!k}[n]\times \comma{z_0}F\longrightarrow
\Lambda^{\!k}[n]\times_{\Delta\C}\comma{[-]}F,
$$
such that the  map induced on diagonals  $$
\Delta\gamma':\Lambda^k[n]\times \gner(\comma{z_0}F)\longrightarrow
\diag\left(\Lambda^k[n]\times_{\gner \C}\Delta(\comma{[-]}F)\right)
$$
is a also a weak homotopy equivalence of simplicial sets.

Since the square of simplicial maps
$$
\xymatrix{\Lambda^k[n]\times \gner(\comma{z_0}F)\ar[d]_{\textstyle
\Delta\gamma'}\ar[r]& \gner[n]\times
\gner(\comma{z_0}F) \ar[d]^{\textstyle \Delta\gamma}\\
\diag(\Lambda^k[n]\times_{\gner
\C}\Delta(\comma{[-]}F))\ar[r]&\diag(\gner[n]\times_{\gner
\C}\Delta(\comma{[-]}F))}
$$
is commutative, the  vertical maps are weak homotopy
equivalences, and the  top map is the trivial cofibration   product of the inclusion $\Lambda^{\!k}[n]\to \Delta[n]$ with the identity map on $\Delta(\comma{z_0}F)$, it follows that the bottom map, that is, that induced on diagonals by (\ref{pub}),
 is a weak homotopy equivalence of simplicial sets, as claimed. This proves (i), but actually we have also shown (ii) since the induced on diagonals by the map (\ref{corner}) is precisely the weak equivalence (\ref{dgamma}) for the case $n=0$.\qed

With the following lemma we will be ready to complete the proof of Theorem \ref{B}. For any $q\geq 0$, forgetting in the second component gives a 2-functor
$$
\Phi:\comma{[q]}{F}\longrightarrow \B,
$$
 $$ \xymatrix @C=7pt@R=10pt{(x,\lv) \ar@/^1pc/[rr]^{\textstyle (u,\y)} \ar@/_1pc/[rr]_-{\textstyle (u',\y')}
& {\Downarrow \alpha} & (x',\lv')}
\xymatrix{\ar@{|->}[r]^{\textstyle \Phi}&} \xymatrix @C=5pt{x
\ar@/^0.7pc/[rr]^{\textstyle u} \ar@/_0.7pc/[rr]_-{\textstyle u'} &
{\Downarrow \alpha} & x'}
$$
and we have an augmented simplicial 2-category
\begin{equation}\label{nphi}\comma{[-]}F \overset{\textstyle \Phi}\longrightarrow \B.
\end{equation}
\begin{lemma}\label{lemphi}
The simplicial $2$-functor $(\ref{nphi})$ induces a homotopy equivalence on classifying spaces,  $\class (\comma{[-]}F)\simeq \class\B$.
\end{lemma}
\proof By Theorem \ref{teo1}, it suffices to prove that the bisimplicial map $$\Delta\Phi:\Delta(\comma{[-]}F)\to\Delta\B$$ induces a weak equivalence $\diag\Delta\Phi:\diag\Delta(\comma{[-]}F)\to\Delta\B$. Now, for every $p\geq 0$, the augmentation
$$
\Delta_p\Phi:\Delta_p(\comma{[-]}F)\!=\!\bigsqcup_{q\geq 0}\!\Delta_p(\comma{[q]}F)\longrightarrow \Delta_p\B
$$
is actually a weak equivalence since, for any geometric simplex $\x:[p]\rightsquigarrow\B$ of $\B$, the fibre $(\Delta_p\Phi)^{-1}(\x)$ is precisely $\Delta(\comma{\C}F\x)$, the geometric nerve of the comma category $\comma{\C}F\x$, which, by Lemma \ref{lemma2} (ii), has a contractible classifying space; that is,
$$|\Delta_p(\comma{[-]}F)|=\bigsqcup_{\x:[p]\rightsquigarrow \B}\class(\comma{\C}F\x)\simeq \bigsqcup_{\x:[p]\rightsquigarrow \B}\hspace{-4pt}pt\simeq \Delta_p\B.$$

Therefore, by Fact \ref{fact}, the map induced on diagonals
$\diag\Delta\Phi$ is also a weak homotopy equivalence, as claimed.
\qed

\vspace{0.3cm}
We can now complete the proof of Theorem \ref{B}:

Consider the diagram of spaces
$$
\xymatrix@R=35pt{\class(\comma{z_0}F)\ar[r]\ar[d]_{\textstyle \class
F'}^{\textstyle \hspace{0.7cm}(a)}&\class (\comma{[-]}F)
\ar[d]_{\textstyle \class F'}^{\textstyle \hspace{0.7cm}(c)}\ar[r]^-{\textstyle \class \Phi}&\class\B
\ar[d]^{\textstyle \class F}\\
\class(\comma{z_0}\C)\ar[r]\ar[d]^{\textstyle
\hspace{0.7cm}(b)}&\class (\comma{[-]}1_\C) \ar[r]_-{\textstyle
\class\Phi}\ar[d]^{\textstyle \class \Psi}&\class\C
\\
pt\ar[r]^{z_0}&\class\C&}
$$
in which the simplicial 2-functor $F':\comma{[-]}F\to \comma{[-]}1_\C$ is  given by
$$
 \xymatrix @C=14pt{(x,\lv) \ar@/^1pc/[rr]^{\textstyle (u,\y)} \ar@/_1pc/[rr]_-{\textstyle (u',\y')}
& {\Downarrow\! \alpha} & (x',\lv')}
\xymatrix{\ar@{|->}[r]^{\textstyle F'}&} \xymatrix @C=1pc{(Fx,\lv)\
\ar@/^1pc/[rr]^{\textstyle (Fu,\y)} \ar@/_1pc/[rr]_-{\textstyle
(Fu',\y')} & {\Downarrow\! F\alpha} & (Fx',\lv')}.
$$

  By Lemma \ref{car}, both squares $(a)+(b)$ and $(b)$ are
homotopy cartesian. It follows that $(a)$ is homotopy cartesian.
Hence, $(a)+(c)=(\ref{hc})$ is also since  both maps $\class\Phi$
are homotopy equivalences by Lemma \ref{lemphi}, whence the theorem.
\qed

\vspace{0.3cm} The following corollary is a direct generalization of
\cite[Theorem A]{quillen}, and it was proved in \cite{b-c}:

\begin{corollary}
Let $F:\B\to \C$ be a $2$-functor between $2$-categories. If the classifying space $\class(\comma{z}F)$ is contractible
for every object $z$ of $\C$, then the induced map $\class F:\class \B\to \class\C$ is a homotopy
equivalence.\end{corollary}

\begin{example}{\em Recall that the classifying space of a strict monoidal category is the classifying space of the one-object 2-category it defines (see examples \ref{e1} and \ref{e2}).  Then, Theorem \ref{B} is applicable to strict monoidal functors (= 2-functors) between strict monoidal categories.

However, we should stress that the homotopy fibre 2-category of a
strict monoidal functor $F:(\m,\otimes)\to(\m',\otimes)$, at the
unique object of the 2-category defined by $(\m',\otimes)$, is not a
monoidal category but a genuine 2-category: Its objects are the
objects $x'\in\m'$, its morphisms are pairs $(x,u'):x'\to y'$ with
$x$ an object in $\m$ and $u':F(x)\otimes x'\to y'$ a morphism in
$\m'$, and its deformations $$ \xymatrix @C=8pt{x'
\ar@/^0.7pc/[rr]^{\textstyle (x,u')} \ar@/_0.7pc/[rr]_-{\textstyle
(y,v')} & {\Downarrow\! u } & y'}
$$
are those morphisms $u:x\to y$ in $\m$ such that the following triangle commutes
$$
\xymatrix@R=4pt{Fx\otimes x'\ar[dd]_{\textstyle Fu\otimes
1_{x'}}\ar[dr]^{\textstyle u'}&\\ &y'.\\Fy\otimes
x'\ar[ru]_{\textstyle v'}&}
$$

In \cite{b-c}, this 2-category was called the {\em homotopy fibre $2$-category} of the monoidal functor $F:(\m,\otimes)\to(\m',\otimes)$, and it was denoted by $\mathcal{K}_F$. Every object $z'$ of $\m'$ determines a 2-endofunctor $-\otimes z':\mathcal{K}_F\to \mathcal{K}_F$, defined by
$$
 \xymatrix @C=8pt{x' \ar@/^1pc/[rr]^{\textstyle (x,u')} \ar@/_1pc/[rr]_-{\textstyle (y,v')}
& {\Downarrow\! u} & y'} \xymatrix{\ar@{|->}[r]^{\textstyle -\otimes
z'}&} \xymatrix @C=8pt{x'\otimes z' \ar@/^1pc/[rr]^{\textstyle
(x,u'\otimes 1_{z'})} \ar@/_1pc/[rr]_-{\textstyle (y,v'\otimes
1_{z'})} & {\Downarrow\! u} & y'\otimes z'},
$$
and, by Theorem \ref{B}, whenever the induced maps $\class(-\otimes z'):\class\mathcal{K}_F\to\class\mathcal{K}_F$, for all $z'\in \mbox{Ob}\m'$, are
homotopy auto-equivalences,  there is a homotopy fibre sequence
$$\class\mathcal{K}_F\to\class(\m,\otimes)\to \class(\m',\otimes).$$

The reader interested in the study of classifying spaces of monoidal categories can find in the above fact a good reason to also be interested in the study of classifying spaces of 2-categories.}
\end{example}

\section{Homotopy fibre sequences induced by lax 2-diagrams}
Theorem \ref{B} can be applied to homotopy theory of lax functors ${\f\!:\!\C^{o}\!\rightsquigarrow \dcat}$, where $\C$ is any 2-category,
hence to  acting monoidal categories, through an enriched {\em
Grothendieck construction}
$$
\xymatrix{\int_\C \!\f,}
$$
that we explain in Definition \ref{gro} below.

Hereafter, $\dcat$ means the 2-category whose objects are 2-categories
$\A$, morphisms the 2-functors $F:\A\to \B$, and deformations $\xymatrix @C=0pc {\A  \ar@/^0.5pc/[rr]^{F} \ar@/_0.5pc/[rr]_{F'} & {}_{\Downarrow\alpha} &\B }$ the 2-natural transformations $\alpha:F\Rightarrow F'$,
as we recalled in the preliminary Section \ref{SP}.
To set some additional notations, let us briefly say that, in the 2-categorical structure of $\dcat$, a $2$-natural transformation $\alpha:F\Rightarrow F'$ composes vertically  with a $2$-natural transformation $\alpha':F'\Rightarrow F''$, that is,
in the category $\dcat(\A,\B)$,  yielding the 2-natural transformation $\alpha\circ\alpha':F \Rightarrow
F''$, whose component at an object $x$ of $\A$ is $(\alpha'\circ\alpha)_x= \alpha'_x\circ\alpha_x$, the horizontal
composition in $\B$ of the morphisms
$$Fx\overset{\alpha_x}\longrightarrow F'x\overset{\alpha'_x}\longrightarrow F''x.$$
Composition of 2-functors $\A\overset{F}\to\B\overset{G}\to\C$, which we are denoting by juxtaposition, that is,
$\A\overset{GF}\to\C$, is the function on objects of the horizontal composition functors $\dcat(\B,\C)\times
\dcat(\A,\B)\to\dcat(\A,\C)$,  which on morphisms of the hom-categories works as follows: for $2$-natural transformations
$\xymatrix @C=0pc {\A  \ar@/^0.5pc/[rr]^{F} \ar@/_0.5pc/[rr]_{F'} & {}_{\Downarrow\alpha} &\B }$ and $\xymatrix @C=0pc {\B  \ar@/^0.5pc/[rr]^{G} \ar@/_0.5pc/[rr]_{G'} & {}_{\Downarrow\beta} &\C }$,  then $\beta\,\alpha:G\,F\Rightarrow G'\,F'$ is the $2$-natural transformation with $(\beta\alpha)_x=\beta_{F'x}\circ G\alpha_x = G'\alpha_x\circ \beta_{Fx}$, the horizontal composite in $\B$ of
the morphisms
$$GFx\overset{G\alpha_x}\longrightarrow GF'x\overset{\beta_{F'x}}\longrightarrow G'F'x.$$

A {\em lax $2$-diagram of $2$-categories with the shape of a
$2$-category $\C$}, or normal lax 2-functor
$$\!\f:\C^{o}\rightsquigarrow \dcat ,$$ provides us with the
following data:
\begin{itemize}
\item a 2-category $\!\f_x$, for each object $x$ of $\C$,
\item a 2-functor $u^*:\!\f_y\to \!\f_x$, for each morphism $u:x\to y$ in $\C$,
\item a 2-natural transformation  $\xymatrix @C=0pc {\f_y  \ar@/^0.7pc/[rr]^{\textstyle u^*}
\ar@/_0.7pc/[rr]_{\textstyle v^*} & {\Downarrow\!\alpha^*} &\f_x}$,
for each deformation $\xymatrix @C=0pc {x
\ar@/^0.5pc/[rr]^{\textstyle u} \ar@/_0.5pc/[rr]_{\textstyle v} &
{\Downarrow\!\alpha} &y }$ of $\C$,
\item a 2-natural transformation  $\xymatrix @C=0.1pc {\f_z  \ar@/^0.7pc/[rr]^{\textstyle v^*u^*}
\ar@/_0.7pc/[rr]_{\textstyle (u\circ v)^*} &
{\Downarrow\!\zeta_{u,v}} &\f_x}$, for each pair of composable
morphisms $x\overset{\textstyle v}\to y\overset{\textstyle u}\to z$
in $\C$,
\end{itemize}
and these data must satisfy the following conditions:
\begin{itemize}
\item for any object $x$ and any morphism $u:x\to y$ in $\C$,
$$1^*_x=1_{\!\f_x} ,\   1_u^*=1_{u^*} \text{ and } \zeta_{u,1_{\!x}}=1_{u^*}=\zeta_{1_{\!y},u},$$
\item for any two vertically composable deformations in $\C$,
$\xymatrix @C=0.7pc {x \ar[rr] \ar@/^1.2pc/[rr]_-{\textstyle
\Downarrow\!\alpha} \ar@/_1.2pc/[rr]^-{\textstyle
\Downarrow\!\alpha'} & &y} $,
$$(\alpha'\alpha)^*=\alpha'^*\circ \alpha^*,$$
\item  for any two horizontally composable deformations in $\C$, $\xymatrix @C=0pc {x  \ar@/^0.7pc/[rr]^{\textstyle
u} \ar@/_0.7pc/[rr]_{\textstyle u'} & {\Downarrow\!\alpha} &y
\ar@/^0.7pc/[rr]^{\textstyle v} \ar@/_0.7pc/[rr]_{\textstyle v'} &
{\Downarrow\!\beta}& z}$, the following diagram of $2$-natural
transformations commutes:
$$
\xymatrix{u^*v^*\ar@{=>}[r]^{\textstyle \zeta_{v,u}}
\ar@{=>}[d]_{\textstyle \alpha^*\beta^*}&(v\circ u)^*
\ar@{=>}[d]^{\textstyle (\beta\circ\alpha)^*}\\
u'^*v'^*\ar@{=>}[r]^{\textstyle \zeta_{v',u'}}&(v'\circ u')^*,}
$$
\item for any three composable morphisms in $\C$,  $x\overset{\textstyle u}\to y\overset{\textstyle v}\to
z\overset{\textstyle w}\to t$, the following diagram of $2$-natural
transformations commutes:
$$
\xymatrix{u^*v^*w^*\ar@{=>}[r]^{\textstyle \zeta_{v,u} w^*}
\ar@{=>}[d]_{\textstyle u^*\zeta_{w,v}}&(v\circ u)^*w^*
\ar@{=>}[d]^{\textstyle \zeta_{w,v\circ u}}\\
u^*(w\circ v)^*\ar@{=>}[r]^{\textstyle \zeta_{w\circ v,u}}&(w\circ
v\circ u)^*.}
$$
\end{itemize}

The so-called Grothendieck construction on a lax functor
${\!\f:\C^{o}\rightsquigarrow \cat}$, for $\C$ a category,
underlies  the following 2-categorical construction, which is
actually a special case of the one considered by Bakovi\'{c} in
\cite[\S 4]{bak} and, for the case where $\C$ is a category, it is
an special case of the ones given by Tamaki in \cite[\S 3]{tama} and
by Carrasco, Cegarra and Garz\'on in \cite[\S 3]{c-c-g}:

\begin{definition} \label{gro}Let $\!\f:\C^{o}\rightsquigarrow \dcat$ be a normal lax functor, where   $\C$ is
a $2$-category.  The Grothendieck construction on $\!\f$ is  the $2$-category, denoted by
$$\xymatrix{\int_\C\!\f,}$$ whose objects are pairs $(a,x)$ where $x$ is an object of $\C$ and $a$
an object of the $2$-category $\f_x$. A morphism $(f,u):(b,y)\to
(a,x)$, is a pair of morphisms where $u:y\to x$ is in $\C$ and
$f:b\to u^*a$ in $\f_y$; and a deformation
$$\xymatrix@C=2pt{(b,y)  \ar@/^1pc/[rr]^{\textstyle (f,u)} \ar@/_1pc/[rr]_{(\textstyle f',u')} &
{\Downarrow\!(\phi,\alpha)} &(a,x) }  $$ consists of a deformation
$\xymatrix @C=3pt{y \ar@/^0.7pc/[rr]^{\textstyle u}
\ar@/_0.7pc/[rr]_{\textstyle u'}
 & {\Downarrow\!\alpha} &x }$ in $\C$,  together with a deformation
$$
\xymatrix@R=15pt@C=15pt{\ar@{}[drr]|(.55)*++{\textstyle \Downarrow\!
\phi}
                & u^*a \ar[dl]_{\textstyle \alpha^*_a}             \\
 u'^*a  & &     b \,,  \ar[ul]_{\textstyle f} \ar[ll]^{\textstyle f'}    }
$$
that is,  $\phi:\alpha^*_a\circ f\Rightarrow f'$, in $\!\f_y$.

The vertical composition in $\int_\C\!\f$ of deformations
$$\xymatrix@C=16pt{(b,y) \ar[rr]_{\textstyle
\Downarrow\!\!(\phi',\alpha')}  \ar@/_1.7pc/[rr]_{\textstyle
(f'',u'')} \ar@/^1.7pc/[rr]^{\textstyle (f,u)}\ar@{}[rr]^{\textstyle
\Downarrow\!(\phi,\alpha)} & &(a,x) }$$ is given by the formula
$(\phi',\alpha')(\phi,\alpha)=(\phi' (1\circ \phi),\alpha'\alpha)$,
where $\alpha'\alpha$ is the vertical composition in $\C$, and $\phi'
(1\circ \phi):(\alpha'\alpha)^*_a\circ f\Rightarrow f''$ is the
deformation in $\!\f_y$ obtained by pasting the diagram
$$
\xymatrix{u''^*a&u'^*a\ar[l]_{\textstyle
\alpha'^*_a}&u^*a\ar[l]_{\textstyle \alpha^*_a}&b.\ar[l]_{\textstyle
f}\ar@/^1.5pc/[ll]_(0.4){\textstyle
\Downarrow\!\phi}|(0.6){\textstyle f'}
\ar@/^2.3pc/[lll]_(0.7){\textstyle
\Downarrow\!\phi'}|(0.6){\textstyle f''}}
$$

The horizontal composition in $\int_\C\!\f$ of two arrows
$$\xymatrix{(c,z)\ar[r]^{\textstyle (g,v)}&
(b,y)\ar[r]^-{\textstyle (f,u)}&(a,x)}$$ is given by $$(f,u)\circ
(g,v) = (\zeta_a\circ v^*\!f\circ g,u\circ v):(c,z)\to (a,x),$$
where $u\circ v$ is the horizontal composition of $u$ and $v$ in
$\C$ and $\zeta_a\circ v^*\!f\circ g$ is the horizontal composite
$$\xymatrix{c\ar[r]^{\textstyle g}&v^*b\ar[r]^{\textstyle v^*\!f}&v^*u^*a\ar[r]^-{\textstyle \zeta_a}&(u\circ v)^*a} $$
in the $2$-category $\!\f_z$; and the horizontal composition of two
deformations
$$\xymatrix@C=3pt{(c,z)  \ar@/^1pc/[rr]^{\textstyle (g,v)} \ar@/_1pc/[rr]_{\textstyle (g',v')} &
{\Downarrow\!(\psi,\beta)} & (b,y)  \ar@/^1pc/[rr]^{\textstyle
(f,u)} \ar@/_1pc/[rr]_{\textstyle (f',u')} &
{\Downarrow\!(\phi,\alpha)} &(a,x) }  $$ is given by the formula
$$
(\phi,\alpha)\circ (\psi,\beta) =\left((1\circ \psi)(1\circ
v^*\!\phi\circ 1), \alpha\circ \beta\right),
$$
where $\alpha\circ \beta$ is the horizontal composition in $\C$ and $(1\circ \psi)(1\circ v^*\phi\circ 1)$ is the deformation in $\!\f_z$ obtained by pasting the diagram
$$
\xymatrix{(u\circ v)^*a \ar[dd]_-{\textstyle
(\alpha\circ\beta)^*_a}\ar@{}[rdd]|{\textstyle =}&
v^*\!u^*\!a\ar@{}[rdd]_(0.5){\textstyle =}\ar[dd]|(0.4){\textstyle
(\beta^*\!\alpha^*\!)_a}\ar[l]_{\textstyle \zeta}
\ar[rd]|{\textstyle v^*\!\alpha^*_a} & &v^*b\ar[ll]_{\textstyle
v^*\!f}\ar[dd]|{\textstyle \beta^*_b}
\ar[ld]|{\textstyle v^*f'}\\
&&v^*u'^*a\ar@{}[rd]|<<<<<{\textstyle =}\ar[ld]|{\textstyle
\beta_{u'^*\!a}}& &c\ar[lu]_{\textstyle g}\ar[ld]^{\textstyle
g'}\ar@{}[lllu]^<<<<<<<{\textstyle \Downarrow\!\psi}
\ar@{}[lllu]|>>>>>>>>>>>>>>>{\textstyle {\Downarrow\! v^*\!\phi}}
\\
(u'\circ v')^*a&v'^*u'^*a\ar[l]^{\textstyle
\zeta}&&v'^*b.\ar[ll]^{\textstyle v'^*\!f'} }
$$
\end{definition}

\begin{remark} {\em We should note that, with the necessary natural changes, the above Grothendieck construction makes sense on lax morphisms of tricategories  $\f:\C^{o}\rightsquigarrow {\mathbf{2\text{-}Cat}}$ with all of its coherence 3-cells invertible, from any $2$-category $\C$ (regarded as a strict tricategory in which the 3-cells are all identities) to the larger tricategory of small 2-categories ${\mathbf{2\text{-}Cat}}$, that is, the full subtricategory given by the 2-categories of the tricategory {\bf Bicat} \cite[\S 5]{g-p-s} of bicategories, pseudo-functors, pseudo-natural transformations, and modifications. However,  the resulting $\int_\C\!\f$ is not a 2-category, but rather a bicategory (see \cite[\S 3]{bak} or \cite[Definition 3.1]{c-c-g} for details.)}\qed
\end{remark}

For any given normal lax functor $\!\f:\C^{o}\rightsquigarrow \dcat
$, the 2-category $\int_\C\!\f$, whose construction is natural both
in $\C$ and $\!\f$, assembles all 2-categories $\!\f_x$ in the
following precise sense: There is a projection 2-functor
$$
\xymatrix{\pi :\int_\C\!\f \to \C,}
$$
given by
$$
\xymatrix@C=0pt{(b,y)  \ar@/^1pc/[rr]^{\textstyle (f,u)}
\ar@/_1pc/[rr]_{\textstyle (f',u')} & {\Downarrow\!(\phi,\alpha)}
&(a,x)}\ \overset{\textstyle \pi}\mapsto \ \xymatrix@C=0pt{y
\ar@/^0.6pc/[rr]^{\textstyle u} \ar@/_0.6pc/[rr]_{\textstyle u'} &
{\Downarrow\!\alpha} &x}
$$
and, for each object $z$ of $\C$, there is a pullback square of 2-categories
\begin{equation}\label{square}
\xymatrix{\f_z\ar[r]^{\textstyle j}\ar[d]&\int_\C\!\f\ar[d]^{\textstyle \pi}\\
[0] \ar[r]^{\textstyle z}&\C}
\end{equation}
where  $j:\!\f_z\to \int_\C\!\f$ is the embedding 2-functor defined by
$$\xymatrix@C=1pt{a \ar@/^0.7pc/[rr]^{\textstyle f} \ar@/_0.7pc/[rr]_{\textstyle g} &
{\Downarrow\!\!\phi} &b}\ \overset{\textstyle j}\mapsto  \
\xymatrix@C=1pt{(a,z)  \ar@/^1pc/[rr]^{\textstyle (f,1_z)}
\ar@/_1pc/[rr]_{\textstyle (g,1_z)} & {\Downarrow\!(\phi,1_{1_z})}
&(b,z) } .$$ Thus, $\!\f_z\cong \pi^{-1}(z)$, which is the fibre
2-category of $\pi$ at $z$.

The following main result in this section is consequence of Theorem \ref{B}:

\begin{theorem}\label{hc2} Suppose that $\f:\C^o\rightsquigarrow \dcat$ is a normal lax functor, where $\C$ is a $2$-category, such that the induced map
$\class w^*:\class \!\f_{z_0}\to \class \!\f_{z_1}$, for each
morphism $w:z_1\to z_0$ in  $\C$,  is a homotopy  equivalence. Then,
for every object $z$ of $\C$, the square induced by $(\ref{square})$
on classifying spaces
\begin{equation}\label{square2}
\xymatrix{\class\!\f_z\ar[r]\ar[d]&\class \!\int_\C\!\f\ar[d]\\ pt
\ar[r]^{\textstyle z}&\class \C}
\end{equation}
is homotopy cartesian.
\end{theorem}

\begin{proof}  Observe that, for each object $z\in \C$, the homotopy fibre 2-category $\comma{z}\pi$ has objects the
triples $(a,x,v)$ with $v:z\to x$ a morphism in $\C$ and $a$ an
object of $\!\f_x$. A morphism $(f,u,\beta):(a,x,v)\to (a',x',v')$
consists of a morphism $u:x\to x'$ together with a deformation
$\beta:u\circ v\Rightarrow v'$ in $\C$ and a morphism $f:a\to u^*a'$
in $\!\f_x$; and a deformation
\begin{equation}\label{defor}
\xymatrix@C=1pt{(a,x,v)  \ar@/^1pc/[rr]^{\textstyle (f,u,\beta)}
\ar@/_1pc/[rr]_{\textstyle (f',u',\beta')} &
{\Downarrow\!(\phi,\alpha)} &(a',x',v') }
\end{equation}
 is a pair with $\alpha:u\Rightarrow u'$ a deformation in $\C$ such that $\beta'
(\alpha\circ 1_v)=\beta$ and
$$
\xymatrix@R=15pt@C=15pt{\ar@{}[drr]|(.55)*+{\textstyle \Downarrow\!
\phi}
                & u^*a' \ar[dl]_{\textstyle \alpha^*_{a'}}             \\
 u'^*a'  & &     a \,,  \ar[ul]_{\textstyle f} \ar[ll]^{\textstyle f'}     }
$$
a deformation in $\f_x$.

We have an embedding 2-functor $\mathbf{i}=\mathbf{i}_z: \!\f_z\to \comma{z}\pi$  given by
$$
\xymatrix@C=1pt{a \ar@/^0.7pc/[rr]^{\textstyle f}
\ar@/_0.7pc/[rr]_{\textstyle f'} & {\Downarrow\!\phi} &a'}\
\overset{\textstyle \mathbf{i}}\mapsto \xymatrix@C=1pt{(a,z,1_z)
\ar@/^1pc/[rr]^{\textstyle (f,1_z,1_{1_z})}
\ar@/_1pc/[rr]_{\textstyle (f',1_z,1_{1_z})} &
{\Downarrow\!(\phi,1_{1_z})} &(a',z,1_z) } ,
$$
and there is also a $2$-functor  $\mathbf{p}=\mathbf{p}_z: \comma{z}\pi\to \!\f_z$ that is defined as follows:
it carries an object $(a,x,v)$ of $\comma{z}\pi$ to the object $v^*a$ of $\!\f_z$,
a morphism $(f,u,\beta):(a,x,v)\to(a',x',v')$ is mapped by $\mathbf{p}$ to the composite morphism
$$v^*a\overset{\textstyle v^*\!f}\longrightarrow v^*u^*a'\overset{\textstyle \zeta_{a'}}\longrightarrow
(u\circ v)^*a' \overset{\textstyle \beta^*_{a'}}\longrightarrow
v'^*a',$$ and a deformation $(\phi,\alpha):(f,u,\beta)\Rightarrow
(f',u',\beta')$, as in (\ref{defor}), to the deformation $1\circ
v^*\phi:\mathbf{p}(f,u,\beta)\Rightarrow \mathbf{p}(f',u',\beta')$
obtained by pasting the diagram
$$\xymatrix@R=17pt@C=25pt{
&v^*u^*a'\ar[r]^{\textstyle \zeta_{a'}}\ar[dd]|(0.3){\textstyle
v^*\alpha^*_{a'}}&
(u\circ v)^*a'\ar[dd]|{\textstyle (\alpha\!\circ \!1_v)^*_{a'}}\ar[rd]^{\textstyle \beta^*_{a'}}&\\
v^*a\ar[rd]_{\textstyle v^*\!f'} \ar[ru]^{\textstyle v^*\!f}
\ar@{}[rrr]|<<<<<<<<<{\textstyle \Downarrow\!
v^*\phi}&\ar@{}[r]|(0.4){\textstyle =}&
\ar@{}[r]|(0.51){\textstyle =}&v'^*a'.\\
&v^*u'^*a'\ar[r]^{\textstyle \zeta_{a'}}&(u'\circ
v)^*a'\ar[ru]_{\textstyle \beta'^*_{a'}}&}
$$

By Lemma \ref{trans}, both 2-functors $\mathbf{i}$ and $\mathbf{p}$ induce homotopy equivalences
on classifying spaces
$$
\class \mathbf{i}:\class\f_z\simeq \class(\comma{z}\pi),
\hspace{0.4cm} \class\mathbf{
p}:\class(\comma{z}\pi)\simeq\class\f_z,$$ since $\mathbf{
p}\,\mathbf{ i} = 1$ and there is a oplax  transformation
$\theta:\mathbf{ i}\,\mathbf{ p}\Rightarrow 1$, whose component at
an object $(a,x,v)$ of $\comma{z}\pi$ is the morphism
$$
\theta_{(a,x,v)}=(1_{v^*a},v,1_v):(v^*a,z,1_z)\to (a,x,v),
$$
and whose component at a morphism $(f,u,\beta):(a,x,v)\to
(a',x',v')$, $$\theta_{(f,u,\beta)}:(f,u,\beta)\circ
\theta_{(a,x,v)}\Rightarrow \theta_{(a',x',v')}\circ
\mathbf{i\,p}(f,u,\beta),$$ is the deformation $(1,\beta)$,
$$
\xymatrix@C=10pt{(v^*a,z,1_z)  \ar@/^1pc/[rr]^{\textstyle
(\zeta_{a'}\circ v^*\!f,u\circ v,\beta)} \ar@/_1pc/[rr]_{\textstyle
(\beta^*_{a'}\circ \zeta_{a'}\circ v^*\!f,v',1_{v'})} &
{\Downarrow\!(1,\beta)} &(a',x',v'). }
$$

The square $(\ref{square})$ is the composite of the squares
$$
\xymatrix{\f_z\ar[r]^{\textstyle
\mathbf{i}_z}\ar[d]^(0.4){\textstyle
\hspace{0.5cm}(a)}&\comma{z}\pi\ar[r]^{\textstyle
\Phi}\ar[d]^(0.4){\textstyle \pi'\hspace{0.3cm}(b)}
&\int_\C\!\f\ar[d]^{\textstyle \pi}\\
[0]\ar[r]^{\textstyle (z,1_z)}&\comma{z}\C\ar[r]^{\textstyle
\Phi}&\C\,,}
$$
where, in $(a)$ both horizontal 2-functors induce homotopy equivalences on classifying spaces (recall Lemma
\ref{cont}). Therefore,  the induced square (\ref{square2}) is homotopy cartesian if and only if the one induced by $(b)$
is as well. But Theorem \ref{B} actually implies that the square induced by $(b)$
$$
\xymatrix{\class(\comma{z}\pi)\ar[r]\ar[d]&\class\! \int_\C\!\f\ar[d]\\ \class(\comma{z}\C)\ar[r]&\class\C}
$$
is homotopy cartesian: to verify the hypothesis of Theorem \ref{B}, let $w:z_1\to z_0$ any given morphism  in
$\C$. We have the square of 2-functors
$$
\xymatrix{\comma{z_0}\pi\ar[r]^{\textstyle \mathbf{p}}\ar[d]_{\textstyle w^*}&\f_{z_0}\ar[d]^{\textstyle w^*}\\ \comma{z_1}\pi\ar[r]^{\textstyle \mathbf{p}}&\f_{z_1}.}
$$
The two composite 2-functors in the square are related by a 2-natural transformation
$w^*\,\mathbf{p}\Rightarrow \mathbf{p}\,w^*$,
whose component at an object $(a,x,v)$ of $\comma{z_0}\pi$ is the
morphism $\zeta_a:w^*v^*a\to(v\circ w)^*a$ in $\comma{\z\xi}F$.
Hence, by Lemma \ref{trans}, the induced square
on classifying spaces
$$
\xymatrix{\class(\comma{z_0}\pi)\ar[r]^{\textstyle \class \mathbf{p}}\ar[d]_{\textstyle \class w^*}&\f_{z_0}\ar[d]^{\textstyle \class w^*}\\ \class(\comma{z_1}\pi)\ar[r]^{\textstyle \class\mathbf{p}}&\f_{z_1}.}
$$
is homotopy commutative; that is, there is a homotopy
$\class w^*\,\class \mathbf{p}\simeq
\class\mathbf{p}\,\class w^*$.  Since the induced maps $\class\mathbf { p}:\class\comma{z_i}\pi\to \class\!\f_{z_i}$, $i=0,1$, and
 $\class w^*:\class\!\f_{z_0}\to \class \!\f_{z_1}$ are all homotopy equivalences, the map $\class
w^*:\class(\comma{z_0}\pi)\to \class(\comma{z_1}\pi)$ is also. \qed \end{proof}

\begin{example}{\em  The $2$-category $\cat$, of small categories, functors and natural transformations, is a full
$2$-subcategory of $\dcat$, regarding  any category as a $2$-category whose deformations are all identities. Hence, the
Grothendieck construction in Definition \ref{gro} works on normal lax functors $\C^{o}\rightsquigarrow\cat$, with $\C$ any $2$-category.
 For any object  $x$ in a $2$-category $\C$, we have the $2$-functor ${\C(-,x):\C^{o}\to \cat}$ on which the
 Grothendieck construction  gives $$\xymatrix{\int_\C\C(-,x)=\comma{\C}x,}$$
 the comma $2$-category of objects over $x$, which, by Lemma \ref{cont}, has a contractible classifying space. Therefore,
  if the object $x$ is such that the induced maps
 $\class u^*: \class\C(z,x)\to\class\C(y,x)$ are homotopy equivalences for the different morphisms $u:y\to z$ in $\C$,
 then Theorem \ref{hc2} implies the existence of a homotopy equivalence $$\Omega(\class\C,x)\simeq \class(\C(x,x)),$$
between the loop space of the classifying space of the $2$-category $\C$ with base point $x$ and the classifying space of the category of
endomorphisms of $x$ in $\C$. }\qed
\end{example}

The well-known Homotopy Colimit Theorem by Thomason \cite{thomason} establishes that the Grothendieck construction
on a diagram of categories is actually a categorical model for the homotopy type of the homotopy colimit of the diagram
of categories. The notion of homotopy colimit has been well generalized in the literature to 2-functors $\!\f:\C^{o}\to
\cat$ (see \cite[2.2]{hinich}, for example), where $\C$ is any 2-category and $\cat\subseteq \dcat$, the 2-category of
small categories, functors and natural transformations. Next, Theorem \ref{hct} generalizes Thomason's theorem both to
2-diagrams of categories $\!\f:\C^{o}\to \cat$, with $\C$ a 2-category,  and to diagrams of 2-categories $\!\f:\C^{o}\to
\dcat$, with $\C$ a category.

Recall that the {\em homotopy colimit} of a 2-functor $\!\f:\C^{o}\to \cat$, where $\C$ is a 2-category,  is defined
\cite[Definition
(2.2.2)]{hinich} to be the simplicial category
\begin{equation}\label{hcc} \mbox{hocolim}_\C\f:\gner^{\!o}\to \cat,
\end{equation}
whose category of $n$-simplices is
$$
 \bigsqcup_{(x_0,\ldots,x_n)\in \mbox{\scriptsize Ob}\C^{n+1}}\hspace{-0.4cm} \!\f_{x_0}\times
\C(x_1,x_0)\times\C(x_2,x_1)\times\cdots\times\C(x_n,x_{n-1});
$$
faces and degeneracies are defined as follows: the face functor $d_0$ maps the component category $\!\f_{x_0}\times \C(x_1,x_0)\times\C(x_2,x_1)\times\cdots\times\C(x_n,x_{n-1})$ into
$\!\f_{x_1}\times \C(x_2,x_1)\times\cdots\times\C(x_n,x_{n-1})$, and it is induced by
$$\xymatrix{
\!\f_{x_0}\times \C(x_1,x_o)\ar[r]^-{\textstyle d_0} &\!\f_{x_1},}
$$
$$
\big(a\overset{\textstyle f}\to b,\xymatrix@C=0pt{x_1 \ar@/^0.7pc/[rr]^{\textstyle u} \ar@/_0.7pc/[rr]_{\textstyle v} & {\Downarrow\!\alpha} &x_0}\big)\ \mapsto
\ \xymatrix@C=35pt{\textstyle u^*a\ar[r]^{\textstyle \alpha^*_b\circ u^*\! f}& v^*b.}
$$
The other face and degeneracy functors  are induced by
the operators $d_i$ and $s_i$ in $\ner\C$ as $1_{\!\f_{x_0}}\times d_i$ and $1_{\!\f_{x_0}}\times s_i$,
respectively. Note that by composing $\mbox{hocolim}_\C\!\f$ with the nerve of categories functor one obtains the {\em bisimplicial Borel construction} $E_\C\f$ in the sense of Tillmann \cite{til}.

\begin{theorem}\label{hct} (i) For any $2$-functor
${\f:\C^{o}\to \cat}$, where $\C$ is a $2$-category, there exists a natural homotopy equivalence $$\xymatrix{\class \,\mbox{\em hocolim}_\C\f\simeq
\class\!\int_\C\!\f,}$$ where $\mbox{\em hocolim}_\C\f$ is given by $(\ref{hcc})$.

(ii) For any  functor $\!\f:\C^{o}\to\dcat$, where $\C$ is a category, there exists a natural homotopy equivalence
$$\xymatrix{\class \,\mbox{\em hocolim}_\C\f\simeq
\class\!\int_\C\!\f,}$$ where $\mbox{\em hocolim}_\C\f\!:=\!\mbox{\em hocolim}_\C\gner\!\f:\Delta^{\!o}\to\set$ is the homotopy
colimit, {\em\cite[Ch. XII]{bousfield-kan}}, of the $\C^{o}$-diagram of simplicial sets obtained by composing $\f$ with the
geometric nerve functor $\gner:\dcat\to \sset$ given by $(\ref{gner})$.

\end{theorem}

\begin{proof} We shall use the bar construction on a bisimplicial set $\overline{W}S$, also called its ``codiagonal" or
``total complex"  \cite{artin-mazur,cordier}. Let us recall that the functor
$$
\overline{W}: \bsset \to \sset
$$
is  the right Kan extension along the ordinal sum functor
$\gner\times \gner \to\gner$, $([p],[q])\mapsto [p+1+q]$. For an
explicit description of the $\overline{W}$ construction, it is often
convenient to view a bisimplicial set $S\!:\!\Delta^{\!o}\!\times\!
\Delta^{\!o}\to \set$ as a (horizontal) simplicial object in the
category of (vertical) simplicial sets. For this case, we write ${d_i^{ h
}\!=\!S(\delta_i,1)\!:\!S_{p,q}\to S_{p-1,q}}$ and $s_i^{ h
}\!=\!S(\sigma_i,1):S_{p,q}\to S_{p+1,q}$ for the horizontal face
and degeneracy maps and, similarly, $d_j^{ v }=S(1,\delta_j)$ and
$s_j^{ v }=S(1,\sigma_j)$ for the vertical ones. Then, for any given
bisimplicial set $S$, $\overline{W}S$ can be described as follows
(cf. \cite[\S III]{artin-mazur}): the set of $p$-simplices of
$\overline{W}S$ is
$$
\Big\{(t_{0}, \dots,t_{p})\in \prod_{m=0}^{p}S_{m,p-m}~|~d^v_0t_{m}=d^h_{m+1}t_{m+1},\, 0\leq m< p\Big\}$$
and, for $0\leq i\leq p$, the faces and degeneracies of a $p$-simplex are given by
$$
\begin{array}{lll}d_i(t_{0},
\dots,t_{p})&=&(d_i^vt_{0},\dots,d_i^vt_{i-1},d_i^ht_{i+1},\dots,d_i^ht_{p}),
\\[-4pt]~\\
s_i(t_{0}, \dots,t_{p})&=&(s_i^vt_{0},\dots,s_0^vt_{i},s_i^ht_{i},\dots,s_i^ht_{p}).
\end{array}
$$

For any bisimplicial set $S$, there is a natural Alexander-Whitney type diagonal approximation, the so-called Zisman comparison map (see \cite{cordier})
\begin{equation}\label{eta} \eta:\diag\,S\to \overline{W}S,
\end{equation}
which carries a $p$-simplex $t\in S_{p,p}$
to
$$
\eta\, t=\big((d_1^h)^pt,(d_2^h)^{p-1}d_0^vt,\dots, (d_{m+1}^h)^{p-m}(d_0^v)^mt,\dots,
(d_0^v)^pt\big).
$$
And the following is a useful result (see  \cite[Theorem 1.1, Theorem 9]{cegarra1,cegarra2}):
\begin{fact} For any bisimplicial set $S$, the simplicial map $\eta:\diag\,S\to \overline{W}S$ is a weak homotopy equivalence.
\end{fact}

\vspace{0.2cm}\noindent\underline{\em Proof of (i)}. The strategy of the proof is to apply the weak homotopy
equivalences (\ref{eta}) above on the following two bisimplicial sets $S_1$ and $S_2$. We let
$$ S_1=\ner\, \text{hocolim}_\C\f:\gner^{\!o}\times\gner^{\!o}\to \set,$$ the bisimplicial set obtained by composing the
simplicial category $\text{hocolim}_\C\f$ (\ref{hcc}) with the nerve functor of categories, and
$$\xymatrix{
S_2=\ner\ner \int_\C\!\f:\gner^{\!o}\times\gner^{\!o}\to \set,}
$$
the bisimplicial set double nerve of the 2-category $\int_\C\!\f$. Since $$\class\,\text{hocolim}_\C\f=|\diag
S_1|,\hspace{0.3cm} \class\! \xymatrix{\int_\C\!\f} =|\diag S_2|,$$  we have homotopy equivalences induced by the maps
(\ref{eta})
$$\class\,\text{hocolim}_\C\!\f\simeq|\overline{W}
S_1|,\hspace{0.3cm} \class\! \xymatrix{\int_\C\!\f} \simeq|\overline{W} S_2|.$$

We shall now observe the nature of the simplicial sets $\overline{W} S_1$ and $\overline{W} S_2$. A
$p$-simplex of $\overline{W} S_1$, say $\chi$, is a list of data
\begin{equation}\label{eqxi}
\chi=\Big(\xymatrix@C=18pt{u_{m}^{m}{}^*a_{m-1}&a_{m}\ar[l]_-{\textstyle f_{m}}},\ \xymatrix@C=2pt{x_{m-1}&  {\Downarrow\!\alpha_{m}^{j}} &x_{m} \ar@/^1pc/[ll]^{\textstyle u_{m}^{j-1}} \ar@/_1pc/[ll]_{\textstyle u_{m}^{j}}}\Big)_{\hspace{-5pt}\scriptsize
\begin{array}{l}0< m\leq p\\[0pt]1< j\leq m,\end{array}}
\end{equation}
in which $x_0,\dots,x_p$ are objects, the $u_{m}^{j}$ morphisms and the $\alpha_{m}^{j}$ are
deformations in $\C$, each $a_m$ is an object in the category $\!\f_{x_m}$ and each $f_{m}$ is a morphism of the
category $\f_{x_{m}}$; while a $p$-simplex $\chi'\in \overline{W} S_2$ is a list of objects, morphisms and
deformations in $\int_\C\!\f$ of the form
\begin{equation}\label{eqxi'}
\chi'=\Big(\xymatrix@C=6pt{(a_{m-1},u_{m-1})  & {\Downarrow\!\alpha_{m}^{j}} &\ar@/^1pc/[ll]^{\textstyle (f_{m}^{j-1}, u_{m}^{j-1})} \ar@/_1pc/[ll]_{\textstyle (f_{m}^{j},u_{m}^{j})} (a_{m},x_{m}) }\Big)_{\hspace{-5pt}\scriptsize
\begin{array}{l}0< m\leq p\\[0pt]1<j\leq m.\end{array}}
\end{equation}

Since in $\chi'$ all triangles
$$
\xymatrix@R=7pt{&u_{m}^{j}{\!}^*a_{m-1}\ar[dd]^{\textstyle (\alpha_{m}^{j})^*_{a_{m-1}}}\\
a_{m}\ar[ru]^(0.4){\textstyle f_{m}^{j}}
\ar[rd]_(0.4){\textstyle f_{m}^{j-1}}&\\&u_{m}^{j-1}{\!}^*a_{m-1}}
$$
must be commutative, we have equalities
$$
f_{m}^{j-1}=(\alpha_{m}^{j})^*_{a_{m-1}}\circ f_{m}^{j}= (\alpha_{m}^{j})^*_{a_{m-1}}\circ
(\alpha_{m}^{j+1})^*_{a_{m-1}}\circ f_{m}^{j+1}=\cdots
$$
$$
\cdots =(\alpha_{m}^{m}\cdots \alpha_{m}^{j})^*_{a_{m-1}}\circ f_{m}^{m}.
$$
Hence, it is straightforward to verify that there is a simplicial bijection $$\overline{W} S_1\cong \overline{W} S_2,$$
that carries a $p$-simplex $\chi$ of $\overline{W} S_1$, as in (\ref{eqxi}), to the $p$-simplex $\chi'$ of $\overline{W} S_2$, as in (\ref{eqxi'}), in which all
data $x's$, $u's$ $\alpha's$ and $a's$ are the same as those in $\chi$, and the morphisms $f_{m}^{j}:a_{m}\to
u_{m}^{j}{\hspace{-1pt}}^*a_{m-1}$ are respectively given by:
$$
\left\{\begin{array}{l}f_{m}^{m}=f_{m},\\[3pt]f_{m}^{j-1}=(\alpha_{m}^{m}\cdots \alpha_{m}^{j})^*_{a_{m-1}}\circ
f_{m} \ \text{ for } 1< j< m.\end{array}\right.
$$
This makes  the proof for $(i)$ complete.

\vspace{0.2cm}\noindent\underline{\em Proof of (ii)}. In this case, we show a weak homotopy equivalence of
simplicial sets $$\text{hocolim}_\C\gner \!\f\to \xymatrix{\gner\int_\C\!\f},$$ and, to do so, we first give a
description of both simplicial sets.

On the one hand, $\text{hocolim}_\C\gner \!\f$ is the simplicial set diagonal of the bisimplicial set
\begin{equation}\label{s}
S=\bigsqcup_{\x\in\ner\C}\gner\!\f_{x_0}=\bigsqcup_{\x:[q]\to \C}\Lfunc([p],\!\f_{x_0}),
\end{equation}
whose $(p,q)$-simplices are pairs
$$
(\y,\x)
$$
consisting of a functor $\x:[q]\to\C$ and a normal lax functor
$\y:[p]\rightsquigarrow \!\f_{x_0}$.  The horizontal
face  and degeneracy maps are given by $$d_i^h(\y,\x)=(\y\delta_i, \x), \hspace{0.2cm}
s_i^h(\y,\x)=(\y\sigma_i, \x),$$ for $0\leq i\leq p$, and the vertical
ones by
$$d_j^v(\y,\x)=(\y, \x\delta_j), \hspace{0.2cm} s_j^v(\y,\x)=(\y, \x\sigma_j),$$
for $0\leq j\leq q$,
except the vertical $0^{\text{th}}$ face which is defined by
$$d_0^v(\y,\x)=(x_{0,1}^*\y, \x\delta_0),$$
where $x_{0,1}^*\y:[p]\rightsquigarrow \!\f_{x_1}$ is the
lax functor obtained by the composition of $\y$ with the 2-functor $x_{0,1}^*:\!\f_{x_0}\to \!\f_{x_{1}}$ attached
in diagram $\f:\C\to \dcat$ at the morphism $x_{0,1}:x_{1}\to x_0$ of $\C$.

On the other hand, a $p$-simplex of $\gner\!\int_\C\!\f$ is a normal lax functor $[p]\rightsquigarrow
\int_\C\!\f$, which can be described as a pair
$$(\yp,\x),$$ where $\x:[p]\to \C$ is a functor,
that is, a $p$-simplex of $\ner\C$, and $\yp:[p]\rightsquigarrow \!\f$ is a normal {\em $\x$-crossed lax functor}
\cite[\S 4.1]{cegarra3}, that is, a family
\begin{equation}\label{xlf}
\yp=\{y'_i,y'_{i,j},y'_{i,j,k}\}_{0\leq i\leq j\leq k\leq p}
\end{equation}
in which each $y'_i$ is an object of the 2-category $\f_{x_i}$, each $y'_{i,j}:y'_j\rightarrow x_{i,j}^*y'_i$ is a
morphism in $\f_{x_j}$, and the $y_{i,j,k}':x_{j,k}^*y'_{i,j}\circ y'_{j,k}\Rightarrow y'_{i,k}$ are deformations in
$\f_{x_{\!k}}$
$$
\xymatrix{\ar@{}[drr]|(.6)*+{\Downarrow y'_{_{i,j,k}}}
                & x_{j,k}^*y'_j \ar[dl]_{\textstyle x_{j,k}^*y'_{i,j}}             \\
 x_{j,k}^*x_{i,j}^*y'_i=x_{i,k}^*y'_i  & &     y'_k\,,   \ar[ul]_{\textstyle y'_{j,k}} \ar[ll]^{\textstyle y'_{i,k}}}
$$
satisfying the condition that, for $0\leq i\leq j\leq k\leq l \leq p$, the following diagram of deformations
\begin{equation}\label{ud}
\xymatrix@C=0pt@R=30pt{x_{k,l}^*(x_{j,k}^*y'_{i,j}\circ y'_{j,k})\circ y'_{k,l}\ar@<-2pt>@{=>}[d]_(0.5){\textstyle
x_{k,l}^*y'_{i,j,k}\circ 1}  &=&
x_{j,l}^*y'_{i,j}\circ x_{k,l}^*y'_{j,k}\circ y'_{k,l}\ar@<2pt>@{=>}[d]^(0.5){\textstyle 1\circ y'_{j,k,l}}&\\  x_{k,l}^* y'_{i,k}\circ y'_{k,l} \ar@{=>}[r]^(0.6){\textstyle y'_{i,k,l}}&
 y'_{i,l}&  x_{j,l}^*y'_{i,j}\circ y'_{j,l} \ar@{=>}[l]_(0.6){\textstyle y'_{i,j,l}}}
\end{equation}
commutes in the category $\!\f_{x_{\!l}}$; and, moreover, the following normalization equations hold: $$y'_{i,i}=1_{y'_i},
\ \ y'_{i,j,j}=1_{y'_{i,j}}=y'_{i,i,j}.$$
The face and degeneracy maps are given by $$d_i(\yp,\x)=(\yp\delta_i, \x\delta_i), \hspace{0.2cm}
s_i(\yp,\x)=(\y\sigma_i, \x\sigma_i),\hspace{0.3cm} \text{for }\ 0\leq i\leq p.$$

  0ur strategy now is to apply the weak homotopy equivalences (\ref{eta}) on the bisimplicial set  $S$,
defined in $(\ref{s})$. Since $\diag S= \text{hocolim}_\C\gner \f$, we have a weak homotopy equivalence $$ \eta:
\text{hocolim}_\C\gner \f\to \overline{W}S$$ and the proof will be complete once we show a simplicial isomorphism
$$\xymatrix{\overline{W}S\cong \gner\!\int_\C\!\f.}$$

For, note that a $p$-simplex of $\overline{W} S$, say $\chi$, can be described as a list of pairs
$$
\chi=\big((\y^{0},\x^{0}), \dots, (\y^{m},\x^{m}),\dots,(\y^{p},\x^{p})\big),
$$
in which each $\x^{m}:[p-m]\to \C$ is a functor and each
$\y^{m}:[m]\rightsquigarrow\!\f_{x^{m}_0}$ is a normal lax functor,
such that $\x^{m}\delta_0=\x^{m+1}$ and $\y^{m+1}\delta_{m+1}={x_{0,1}^{m\,*}}\y^{m}$,  for all $0\leq m<p$.
Denoting $\x^{0}:[p]\to \C$ simply by $\x:[p]\to \C$, an iterated use of the above equalities proves
that
$$
\x^{m}=\x(\delta_0)^m:[p-m]\overset{(\delta_0)^m}\to [p]\overset{\x}\to \C,
$$
for $0\leq m\leq p$, and
$$
\y^{m}\delta_{m}\cdots\delta_{k+1}=x_{k,m}^{\ *}\y^{k}:[k]\rightsquigarrow \!\f_{x_{m}},$$ for $0\leq k< m\leq p$.
These latter equations mean that  $$ \left\{\begin{array}{ll}y^{j}_i=x_{i,j}^*\,y^{i}_i& \text{for } \ i\leq j,\\[4pt]
y^{k}_{i,j}=x_{j,k}^*\,y^{j}_{i,j}& \text{for } \ i\leq j\leq k,\\[4pt]
y^{l}_{i,j,k}=x_{k,l}^*\ y^{k}_{i,j,k}& \text{for }\  i\leq j\leq k\leq l,
\end{array}\right.
$$ whence we see how the $p$-simplex $\chi$ of $\overline{W} S$ is uniquely determined by $\x:[p]\to\C$, the objects
$y^{i}_i$ of $\f_{x_i}$,  the morphisms $y^{j}_{i,j}:y^{j}_j\to y^{j}_i=x_{i,j}^*y^{i}_i$ of $\f_{x_j}$ and
the deformations $y^{k}_{i,j,k}:y^{k}_{i,j}\circ y^{k}_{j,k}=x_{j,k}^*y^{j}_{i,j}\circ y^{k}_{j,k}\Rightarrow
y^{k}_{i,k}$ in $\f_{x_k}$, all for $0\leq i\leq j\leq k\leq p$. At this point, we observe that there is a normal
$\x$-crossed lax functor
$\yp=\{y'_i,y'_{i,j},y'_{i,j,k}\}:[p]\rightsquigarrow \!\f$, as in $(\ref{xlf})$, defined just
by stating  $y'_i=y^{i}_i$, $y'_{i,j}=y^{j}_{i,j}$ and $y'_{i,j,k}=y^{k}_{i,j,k}$ (the commutativity of diagrams
$(\ref{ud})$ follows from $\y^{l}$ being a lax functor). Thus, the $p$-simplex $\chi\in \overline{W}S$ defines the
$p$-simplex  $(\yp,\x)$ of $\gner\!\int_\C\!\f$, which itself uniquely determines  $\chi$.
In this way, we obtain an injective simplicial map $$\xymatrix{j:\overline{W}S\to  \gner\!\int_\C\!\f}$$
$$\big((\y^{0},\x^{0}), \dots, (\y^{p},\x^{p})\big)\overset{\textstyle j}\mapsto (\yp,\x)=
(\{y^{i}_i,y^{j}_{i,j},y^{k}_{i,j,k}\}, \x^{0}),$$ which is also surjective, that is, actually an isomorphism,
as we can see by retracing our steps:
 To any pair $(\yp,\x)$ describing a $p$-simplex of $\gner\int_\C\!\f$, that is, with $\x=\{x_i,x_{i,j}\}:[p]\to \C$ a functor and $\yp=\{y'_i,y'_{i,j},y'_{i,j,k}\}
 :[p]\rightsquigarrow\!\f$ a normal $\x$-crossed lax
 functor, we associate the $p$-simplex
 $\chi=\big((\y^{m},\x^{m})\big)$ of $\overline{W}S$, where, for each $0\leq m\leq $,
 $\x^{m}:[p-m]\to \C$ is the
 composite $[p-m]\overset{(\delta_0)^{m}}\to [p] \overset{\x}\to\C$,
and the normal lax functor $\y^{m}:[m]\rightsquigarrow \!\f_{\!x^{m}_0}=\!\f_{x_m}$ is defined by the objects
$y^{m}_i=x_{i,m}^*y'_i$, the morphisms $y^{m}_{i,j}=x_{j,m}^*y'_{i,j}:y^{m}_{j}\to y^{m}_{i}$ and the
deformations $y^{m}_{i,j,k}=x_{k,m}^*y'_{i,j,k}:y^{m}_{i,j}\circ y^{m}_{j,k}\Rightarrow y^{m}_{i,k}$. Since one
easily checks that $j(\chi)=(\yp,\x)$, the proof is complete. \qed
\end{proof}

\begin{remark} {\em If $\f:\C^{o}\to \cat$ is any  $2$-functor from a $2$-category $\C$ such that any morphism  $z_1\to z_0$ in $\C$ induces a homotopy equivalence
$\class \f_{z_0}\to \class \f_{z_1}$,  then it is a consequence of Theorems \ref{hc2} and \ref{hct} that, for each object $z$ of $\C$,  there is a homotopy cartesian induced square
$$
\xymatrix{\class\f_z\ar[r]\ar[d]&\class \mbox{hocolim}_\C\!\f\ar[d]\\ pt \ar[r]^{\textstyle z}&\class \C.}
$$
This is a fact that, alternatively, can be obtained from a general result on simplicial categories acting on simplicial sets by Moerdijk \cite[Theorem 2.1]{moer}. The analogous result for a functor $\!\f:\C^{o}\to \dcat$, where $\C$ is a category, can also be obtained directly from Quillen's Lemma \cite[p. 90]{quillen}, \cite[\S IV, Lemma 5.7]{g-j}.}\qed
\end{remark}

\begin{example}{\em
Let $(\m,\otimes)$ be a strict monoidal category. If we regard the monoidal category
 as a $2$-category with only one object, say $1$, then we can identify a
$2$-functor $${\mathcal{N}:(\m,\otimes)^{o}\to \cat}$$ with a category $\mathcal{N}$ ($=\mathcal{N}1$, the one associated to
the unique object of the $2$-category) endowed with an associative and unitary right action of $\m$  by a functor
$\otimes :\mathcal{N}\times \m\to \mathcal{N}$; namely, that given by $$(a\overset{\textstyle f}\to b)\otimes (u\overset{\textstyle \alpha}\to v)=
(u^*a\overset{\textstyle \alpha^*_b\circ u^*\!f}\longrightarrow v^*b).$$

Since there is an identification of simplicial categories
$$\mbox{hocolim}_{(\m,\otimes)}\mathcal{N}= E_{(\m,\otimes)}\mathcal{N},$$ where
$E_{(\m,\otimes)}\mathcal{N}:\gner^{\!o}\to\cat$, $[n]\mapsto \mathcal{N}\times \m^n$, is the simplicial category obtained by the so-called Borel construction (or bar
construction) for the action, it follows from Theorem \ref{hct} that $\int_{(\m,\otimes)}\mathcal{N}$ is a $2$-category modelling the homotopy type of
the Borel simplicial category $ E_{(\m,\otimes)}\mathcal{N}$, that is, there is a homotopy equivalence
$$\xymatrix{\class \int_{(\m,\otimes)}\mathcal{N}\simeq \class E_{(\m,\otimes)}\mathcal{N}.}$$

This $2$-category $\int_{(\m,\otimes)}\mathcal{N}$ has the following easy description: its objects are the same as
$\mathcal{N}$. A morphism $(f,u):a\to b$ in $\int_{(\m,\otimes)}\mathcal{N}$ is a pair $(f,u)$ with $u$ an object of $\m$
and $f:a\to b\otimes u$ a morphism in $\mathcal{N}$, and a deformation $\xymatrix @C=6pt {a  \ar@/^0.7pc/[rr]^{\textstyle (f,u)} \ar@/_0.7pc/[rr]_{\textstyle (g,v)} & {\Downarrow\!\alpha} &b }$ is a
morphism $\alpha:u\to v$ in $\m$ such that the following triangle commutes
$$
\xymatrix@C=12pt@R=16pt{&a\ar[rd]^{\textstyle g}\ar[ld]_{\textstyle f}&\\b\otimes u\ar[rr]^{\textstyle 1\otimes \alpha}&&b\otimes v.}
$$
The compositions in $\int_{(\m,\otimes)}\mathcal{N}$ are given in an obvious manner.

Many of the homotopy-theoretical properties of the classifying space of a monoidal category, $\class(\m,\otimes)$, can now be
more easily reviewed by using Grothendieck $2$-categories $\int_{(\m,\otimes)}\mathcal{N}$ instead of the Borel
simplicial categories $E_{(\m,\otimes)}\mathcal{N}$.

Thus, one sees, for example, that if the action  is
such that multiplication by each object $u$ of $\m$, $a\mapsto a\otimes u$, induces a homotopy equivalence $\class
\mathcal{N}\overset{\sim}\to\class \mathcal{N}$, then, by Theorem \ref{hc2}, $\class \mathcal{N}$ is homotopy
equivalent to the homotopy fibre of the map ${\class \int_{(\m,\otimes)}\mathcal{N}\to \class (\m,\otimes)}$ (cf.
\cite[Proposition 3.5]{jardine}); that is, one has a homotopy fibre sequence
$$\xymatrix{\class \mathcal{N}\to \class \!\int_{(\m,\otimes)}\mathcal{N}\to \class (\m,\otimes).}$$

In particular, the right action of $(\m,\otimes)$ on the underlying category $\m$ leads to the
$2$-category $\int_{(\m,\otimes)}\mathcal{M}= \comma{1}(\m^{\!o},\otimes)$, the comma $2$-category whose classifying space
is contractible by Lemma \ref{cont} (cf. \cite[Proposition 3.8]{jardine}). Then, it follows the well-known fact that there is
a homotopy equivalence $$\class \m\simeq \Omega\class(\m,\otimes),$$ between the classifying space of the underlying
category and the loop space of the monoidal category, whenever multiplication for each object $u\in \m$, $v\mapsto v\otimes u$, induces a
homotopy autoequivalence on $\class\m $.
}\end{example}

\begin{example} {\em  Let us recall that the category of simplices  of a simplicial set $S:\Delta^{\!o}\to \set$, $\int_\Delta {\hspace{-3pt}}S$, has as objects pairs
$(x,m)$ where $m\geq 0$ and $x$ is a $m$-simplex of $S$; and arrow $\xi\!:(x,m)\rightarrow (y,n)$ is an arrow
$\xi\!:[m]\to [n]$ in $\Delta$ with the property $x=\xi^*y$. It is a well-known result, due to Illusie \cite[Theorem 3.3]{illusie}, that there exists a homotopy equivalence $|S|\simeq
\class\!\int_\Delta {\hspace{-2pt}}S$ between the geometric realization of $S$ and the classifying space of $\int_\Delta{\hspace{-3pt}}S$. This result is,
in fact, a very particular case of the homotopy colimit theorem of Thomason \cite{thomason}): If $\C:\Delta^{\!o}\to \cat$ is any simplicial category, then there is a homotopy equivalence $\class \C\simeq \class\!\int_\Delta\!\C$, where $\int_\Delta\!\C$ is the category Grothendieck construction on $\C$.

Now, from Theorem \ref{hct} (ii), the homotopy type of any given simplicial $2$-category $\C:\Delta^{\!o}\to \dcat$, $[n]\mapsto \C_n$, is the same as the homotopy type of the $2$-category $\int_\Delta\!\C$, that is, $\class \C\simeq \class\!\int_\Delta\!\C$. To describe this $2$-category, note that its set of objects  is
$$ \xymatrix{\text{Ob}\!\int_\Delta\!\C}=\bigsqcup_{n\geq 0}\text{Ob}\C_n,$$ and its hom-categories are $$\xymatrix{\int_\Delta\!\C}\!\big((x,m),(y,n)\big)=\bigsqcup_{[m]\overset{\xi}\to[n]}
\C_m(x,\xi^*y),$$
where the disjoint union is taken over all maps $[m]\to [n]$ in $\Delta$.

For instance, if $\C$ is any $3$-category, that is, a $\dcat$-enriched category, then $\ner \C$ is a simplicial $2$-category whose classifying space is the classifying space of the $3$-category. Therefore, we have homotopy equivalences $$\xymatrix{\class\C\simeq\class \int_\Delta\!\ner\C\simeq \class \int_\Delta\!\ner(\int_\Delta\!\ner\C)}.$$}\qed
\end{example}

\section*{Acknowledgement}
\begin{enumerate}
\item The author would like to express his gratitude to the referee, whose comments greatly improved this paper.
\item Partially supported by DGI of Spain and FEDER (Project: MTM2007-65431); Consejer\'{i}a de Innovacion
de J. de Andaluc\'{i}a (P06-FQM-1889); MEC de Espa\~{n}a, `Ingenio Mathematica(i-Math)' No. CSD2006-00032
(consolider-Ingenio 2010).
\end{enumerate}


\begin{thebibliography}{99}

\bibitem{artin-mazur} M. Artin, B. Mazur,   On the Van Kampen Theorem, Topology 5, 179-189 (1966).

\bibitem{baez} {J.C. Baez, A.D. Lauda,} Higher-Dimensional Algebra V:
2-Groups, Theory Appl. Categ. 12, 423-491 (2004).


\bibitem{bak} I. Bakovi\'{c}, Grothendieck construction for bicategories, preprint available at http://www.irb.hr/korisnici/ibakovic/sgc.pdf


\bibitem{borceux} F. Borceux, Handbook of Categorical Algebra 1. Basic category theory,  Encyclopedia of Mathematics and its Applications 50, Cambridge University Press, Cambridge (1994).

\bibitem{bousfield-kan} A. K. Bousfield, D. M. Kan,  Homotopy limits, completions and localizations, Lecture Notes in Math. 304, Springer-Verlag, Berlin-New York (1972).

\bibitem{b-s} {R. Brown, C.B. Spencer,} G-groupoids, crossed modules and the fundamental groupoid of a topological group, Proc. Kon. Nederl. Acad. Wetensch. 79, 296-302 (1976).

\bibitem{b-c} M. Bullejos, A. M. Cegarra,  On the geometry of 2-categories and their classifying spaces, $K$-Theory 29, no. 3, 211-229 (2003).


\bibitem{cordier} J.-M. Cordier, Sur les limites homotopiques de diagrammes homotopiquement coh\'{e}rents,
Compositio Math.  62, 367-388 (1987).


\bibitem{duskin} J. W. Duskin, Simplicial matrices and the nerves of weak $n$-categories. I. Nerves of bicategories. CT2000 Conference (Como). Theory Appl. Categ. 9, 198-308 (2001/02).

\bibitem{c-c-g} P. Carrasco, A. M. Cegarra, A. R. Garz\'on, Classifying spaces for braided monoidal categories and lax diagrams of bicategories, to appear in Adv. Math.,  available at http://arxiv.org/abs/0907.0930v1.

\bibitem{c-g-o} A. M. Cegarra, J. M.  Garc\'{i}a-Calcines, J. M.; Ortega, On graded categorical groups and equivariant group extensions. Canad. J. Math. 54, 970-997 (2002).

\bibitem{cegarra1} A. M. Cegarra, J. Remedios, The relationship between the diagonal and the bar constructions on a bisimplicial set. Topology Appl. 153, no. 1, 21-51 (2005).


\bibitem{cegarra2} A. M. Cegarra,  J. Remedios, The behaviour of the $\overline W$-construction on the homotopy theory of bisimplicial sets. Manuscripta Math. 124, no. 4, 427-457 (2007).

\bibitem{cegarra3} A. M. Cegarra, E. Khadmaladze, Homotopy classification of graded Picard categories. Adv. Math. 213, no. 2, 644-686 (2007).

\bibitem{fiedorowicz}  Z. Fiedorowicz,  Classifying spaces of topological monoids and categories,  Amer. J. Math. (2) 106,  301-350 (1984).

\bibitem{g-j} P. G. Goerss, J. F. Jardine, Simplicial homotopy theory. Progress in Mathematics 174, Birkh\"{a}user Verlag, Basel (1999).

\bibitem{g-p-s} R. Gordon, A. J.  Power, R. Street,  Coherence for tricategories, Mem. Amer. Math. Soc. 117 (558) (1995).

\bibitem{gray} J. W. Gray, Closed categories, lax limits and homotopy limits, J. Pure Appl. Algebra  19, 127-158 (1980).

\bibitem{groth} A. Grothendieck, Techniques de construction et th\'{e}or\`{e}mes d'existence en g\'{e}eom\'{e}trie alg\'{e}rique. III. Pr\'{e}schemas quotients, S\'{e}minaire Bourbaki, 13e ann\'{e}e, 1960/61, no 212, F\'{e}vrier 1961.

\bibitem{grothendieck} A. Grothendieck, Cat\'{e}gories fibr\'{e}es et d\'{e}scente, SGA I,
expos\'{e} VI, Lecture Notes in Math. 224, 145-194,  Springer, Berlin (1971).

\bibitem{hinich} V. A. Hinich, V. V. Schechtman: Geometry of a category of complexes and algebraic
K-theory, Duke Math. J. 52, 339-430 (1985).

\bibitem{illusie} L. Illusie: Complexe cotangent et d\'{e}formations II,  Lecture Notes in Math.  283, Springer-Verlag (1972).

\bibitem{jardine} J. F. Jardine,  Supercoherence,   J. Pure Appl. Algebra 75, 103-194 (1991).


\bibitem{quillen67} D. Quillen, Homotopical Algebra, Lecture Notes in Math. 43, Springer-Verlag (1967).

\bibitem{quillen} D. Quillen, Higher algebraic K-theory:I, in {\em Algebraic K-theory, I: Higher K-theories}, Lecture
Notes in Math. 341, 85-147, Springer-Verlag  (1973).

\bibitem{macd} D. McDuff, On the classifying spaces of discrete monoids, Topology  18, 313-320 (1979).


\bibitem{maclane} S. Mac Lane,  Categories for the working mathematician, GTM {\bf 5} 2nd Edition,
Springer (1998).

\bibitem{mac} S. Mac Lane and J.H.C. Whitehead,  On the 3-type of a
complex,  Proc. Nat. Acac. Sci. USA 30, 41-48 (1956).


\bibitem{may67} J. P. May, {\em Simplicial objects in algebraic topology},  viii+161 pp., Reprint of the 1967 original. Chicago Lectures in Mathematics. University of Chicago Press, Chicago, IL (1992).

\bibitem{milnor} J. Milnor, The geometric realization of a semi-simplicial complex,
Ann. of Math. (2) 65 (1957),  357-362.

\bibitem{moer} I. Moerdijk, Bisimplicial sets and the group-completion theorem. In {\em  Algebraic
K-theory: connections with geometry and topology (Lake Louise,
AB, 1987)}, 225-240, Kluwer Acad. Publ., Dordrecht (1989).

\bibitem{moerdijk-svensson} I. Moerdijk, J. A. Svensson, Algebraic classification of equivariant
homotopy 2-types, I, J. Pure Appl. Algebra 89, 187-216 (1993).

\bibitem{segal68} G. B. Segal, Classifying spaces and spectral sequences, Publ. Math. Inst. des
Hautes Etudes Scient. (Paris) 34, 105-112 (1968).

\bibitem{street} R. Street, The algebra of oriented simplices,  J. Pure Appl. Algebra   49, no. 3, 283-335 (1987).

\bibitem{street2} R. Street, Categorical structures. Handbook of algebra, Vol. 1, 529-577, North-Holland, Amsterdam (1996).

\bibitem{tama} D. Tamaki, The Grothendieck Construction and Gradings for
Enriched Categories, preprint available at http://arxiv.org/abs/0907.0061v1.

\bibitem{thomason} R. W. Thomason, Homotopy colimits in the category of small categories,
Math. Proc. Cambridge Philos. Soc. 85, no. 1, 91-109 (1979).


\bibitem{til}  U. Tillmann,  On the homotopy of the stable mapping class group. Invent.
Math.  130, no. 2, 257-275 (1997).

\bibitem{til2} U.  Tillmann, Discrete models for the category of Riemann surfaces. Math. Proc. Cambridge Philos. Soc. 121, no. 1, 39-49 (1997).

\bibitem{whi} {J.H.C.  Whitehead,} Combinatorial homotopy II, Bull. A.M.S. 55,  496-543 (1949).


\bibitem{tonks} K. Worytkiewicz, K. Hess, P. E.  Parent, A. Tonks, A model structure \`{a} la Thomason on 2-Cat, Journal of Pure and Applied Algebra 208, no. 1, 205-236 (2007).

\end{thebibliography}
\end{document}